\documentclass[a4paper]{article}

\usepackage{amsmath}
\usepackage{amssymb}
\usepackage{amsfonts}
\usepackage{theorem}
\usepackage[utf8]{inputenc}
\usepackage[all]{xypic}
\usepackage{hyperref}

\newtheorem{theorem}{Theorem}[section]
\newtheorem{proposition}[theorem]{Proposition}
\newtheorem{lemma}[theorem]{Lemma}
\newtheorem{corollary}[theorem]{Corollary}

\newcommand{\elemsk}[2]{[#2]_{#1}}
\newcommand{\baseM}[1]{\mathcal{M}_{#1}}
\newcommand{\baseN}[1]{\mathcal{N}_{#1}}
\newcommand{\PP}[2]{\mathcal{P}(#1, #2)}
\newcommand{\PPd}[3]{\mathcal{P}(#1, #2)_{#3}}
\newcommand{\id}[1]{\mathrm{id}_{#1}}

\newcommand{\dclass}[1]{\mathcal{D}_{#1}}
\newcommand{\DD}{\mathcal{D}}
\newcommand{\D}[1]{#1/\DD}
\newcommand{\LL}{\mathcal{L}}
\newcommand{\RR}{\mathcal{R}}
\newcommand{\Drel}{\mathbin{\DD}}
\newcommand{\Rrel}{\mathbin{\RR}}
\newcommand{\Lrel}{\mathbin{\LL}}

\newcommand{\im}[1]{\mathrm{im}(#1)}
\newcommand{\thetaeq}{\mathbin{\theta}}

\newcommand{\set}[1]{\{#1\}}
\newcommand{\such}{\mid}

\newcommand{\Sk}[1]{\mathsf{Sk}(#1)}
\newcommand{\St}[1]{\mathsf{St}(#1)}

\newcommand{\up}[1]{{#1}_\sharp}
\newcommand{\down}[1]{{#1}_\flat}
\newcommand{\homo}[1]{{#1}_\natural}

\newcommand{\dom}[1]{\mathrm{dom}(#1)}
\newcommand{\restricted}[2]{#1{\mid}_{#2}}

\newcommand{\leqideal}[1]{\langle #1 \rangle_{\leq}}
\newcommand{\sand}{\curlywedge}
\newcommand{\sor}{\curlyvee}

\newenvironment{proof}{\medskip\emph{Proof.}}{\hfill$\Box$\medskip}

\newcommand{\three}{\mathsf{3}}
\newcommand{\two}{\mathsf{2}}
\newcommand{\one}{\mathsf{1}}

\newcommand{\GBA}{\mathsf{GBA}}
\newcommand{\BA}{\mathsf{BA}}

\newcommand{\op}[1]{#1^{\mathrm{op}}}

\newcommand{\CatA}{\mathsf{SkAlg}} 
\newcommand{\CatRA}{\mathsf{SkAlg_R}} 
\newcommand{\CatS}{\mathsf{SkSp}} 
\newcommand{\CatRS}{\mathsf{SkRSp}} 
\newcommand{\FunS}{\mathcal{S}} 
\newcommand{\FunA}{\mathcal{A}} 

\newcommand{\parto}{\rightharpoonup}

\newcommand{\pbcorner}[1][dr]{\save*!/#1-1.2pc/#1:(-1,1)@^{|-}\restore}

\newdir{ >}{{}*!/-10pt/@{>}} 
\newdir{ (}{{}*!/-5pt/@^{(}} 
\newdir{|>}{!/5pt/@{|}*:(1,-.2)@{>}*:(1,+.2)@_{>}}

\begin{document}
\title{Stone Duality for\\ Skew Boolean Algebras with Intersections}
\author{Andrej Bauer\\
\small Faculty of Mathematics and Physics\\
\small University of Ljubljana\\
\small \texttt{Andrej.Bauer@andrej.com}
\and
Karin Cvetko-Vah\\
\small Faculty of Mathematics and Physics\\
\small University of Ljubljana\\
\small \texttt{Karin.Cvetko@fmf.uni-lj.si}
}
\maketitle

\begin{abstract}
  We extend Stone duality between generalized Boolean algebras and Boolean
  spaces, which are the zero-dimensional locally-compact Hausdorff spaces, to a
  non-commutative setting. We first show that the category of right-handed skew
  Boolean algebras with intersections is dual to the category of surjective
  étale maps between Boolean spaces. We then extend the duality to skew Boolean
  algebras with intersections, and consider several variations in which the
  morphisms are restricted. Finally, we use the duality to construct a
  right-handed skew Boolean algebra without a lattice section.
\end{abstract}

\section{Introduction}
\label{sec:introduction}

The fundamental example of the kind of duality we are interested in was
established by Marshall Stone~\cite{stone36:_boolean,stone37:_applic_boolean}:
every Boolean algebra corresponds to a zero-dimensional compact Hausdorff space,
or a \emph{Stone space} for short, as well as to a \emph{Boolean ring}, which is
a commutative ring of idempotents with a unit. In modern language the duality is
stated as equivalence of categories of Boolean algebras, Boolean rings, and
Stone spaces, where the later equivalence is contravariant. The duality has many
generalizations, see~\cite{johnstone82:_stone}. Already in Stone's second
paper~\cite[Theorem~8]{stone37:_applic_boolean} we find an extension of duality
to \emph{Boolean spaces}, which are the zero-dimensional \emph{locally} compact
Hausdorff spaces. They correspond to commutative rings of idempotents, possibly
without a unit, or equivalently to \emph{generalized Boolean algebras}, which
are like Boolean algebras without a top element.

Our contribution to the topic is a study of the \emph{non-commutative} case.
Among several variations of non-commutative Boolean algebras we are able to
provide duality for \emph{skew Boolean algebras with intersections} (which we
call \emph{skew algebras}) because they have a well-behaved theory of ideals.

The paper is organized as follows. In Section~\ref{sec:preliminaries} we recall
the necessary background material about skew Boolean algebras, Boolean spaces,
and étale maps. In Section~\ref{sec:duality-commutative} we spell out the
well-known Stone duality for commtutative algebras. In
Section~\ref{sec:duality-skew} we establish the duality between right-handed
skew algebras and skew Boolean spaces, which we then extend to the duality
between skew algebras and rectangular skew Boolean spaces. In
Section~\ref{sec:variations} we further analyze the situation and consider
several variations of the duality in which morphisms are restricted. In
Section~\ref{sec:lattice-sections} we use the duality to construct a
right-handed skew algebra without a lattice section. This answers negatively a
hitherto open question of existence of such algebras.

\paragraph*{Acknowledgment.}

We thank Jeff Egger, Mai Gehrke, Ganna Kudryavtseva, Jonathan Leech, and Alex
Simpson for discussing the topic with us and offering valuable advice.

\section{Preliminary definitions}
\label{sec:preliminaries}

In the first part of the section we review basic concepts and notation
regarding skew Boolean algebras. In the second part we recall some
basic facts about Boolean spaces and étale maps.

\subsection{Skew Boolean algebras}
\label{sec:skew-boolean-algebras}

A \emph{skew lattice} is an algebra $(A, {\land}, {\lor})$ with idempotent and
associative binary operations \emph{meet~$\land$} and \emph{join~$\lor$}
satisfying the absorption identities $x \land (x \lor y) = x = (y \lor x) \land
x $ and $x \lor (x \land y) = x = (y \land x) \lor x$. If one of the operations
is commutative then so is the other, in which case~$A$ is a lattice,
see~\cite{L1}.

A skew lattice has two order structures. The \emph{natural partial order} $x
\leq y$ is defined by $x \land y = y \land x = x$, or equivalently $x \lor y = y
\lor x = y$. The \emph{natural preorder} $x \preceq y$ is defined by $x \land y
\land x = x$, or equivalently $y \lor x \lor y = y$. The poset refelection of
the natural preorder $\preceq$ is known as \emph{Green's relation~$\DD$}. By
Leech's First Decomposition Theorem~\cite{L1}, $\DD$ is the finest congruence
for which $\D{A}$ is a lattice. In other words, the functor $A\mapsto \D{A}$ is
a reflection of skew lattices into ordinary lattices. We denote the
$\DD$-equivalence class of $a$ by $\dclass{a}$.

The reflection can be analysed further into its left- and right-handed
parts. A skew lattice $A$ is \emph{right-handed} if it satisfies the
identity $x \land y \land x = y \land x$, or equivalently $x \lor y
\lor x = x \lor y$. We define \emph{left-handed} lattices analogously.
Quotients by Green's congruence relations $\RR$ and $\LL$, which are defined by
\begin{align*}
  x \Rrel y &\iff x \land y = y \mathbin{\text{and}} y \land x = x,\\
  x \Lrel y &\iff x \land y = x \mathbin{\text{and}} y \land x = y,
\end{align*}
provide reflections of a skew lattice~$A$ into left-handed and right-handed skew
lattices, respectively. By Leech's Second Decomposition Theorem~\cite{L1} the
square of canonical quotient maps
\begin{equation*}
  \xymatrix@+1em{
    A \ar[d]  \ar@{->}[r]
    &
    A/\LL \ar[d]
    \\
    A/\RR \ar@{->}[r]
    &
    A/\DD   
  }
\end{equation*}
is a pullback in the category of skew algebras.

A \emph{rectangular band $(A, {\land})$} is an algebra with a binary
operation~$\land$ which is idempotent, associative, and it satisfies the
rectangle identity $x \land y \land z = x \land z$. The name comes from the fact
that every rectangular band is isomorphic to a cartesian product $X \times Y$
with the operation $(x_1, y_1) \land (x_2, y_2) = (x_1, y_2)$. If $A$ is
non-empty the sets $X$ and $Y$ are unique up to bijection and can be taken to be
$A/\LL$ and $A/\RR$, respectively. Each rectangular band is a skew lattice
for~$\land$ and the associated operation~$\lor$ defined as $x \lor y = y \land
x$. It turns out that the $\DD$-classes of a skew lattice form rectangular
bands, with the operation induced by the skew lattice.

A \emph{skew Boolean algebra} $(A, 0, {\land}, {\lor}, {\setminus})$ is a
skew lattice which is meet-distributive, i.e., it satisfies the identities
\begin{equation*}
  x \land (y \lor z) = (x \land y) \lor (x \land z)
  \quad\text{and}\quad
  (y \lor z) \land x = (y \land x) \lor (z \land x),
\end{equation*}
has a \emph{zero~$0$}, which is neutral for~$\lor$, and a \emph{relative
  complement~$\setminus$} satisfying
\begin{equation*}
  (x \setminus y) \land (x \land y \land x) = 0
  \quad\text{and}\quad
  (x \setminus y) \lor (x \land y \land x) = x.
\end{equation*}
It follows from these requirement that the principal subalgebras $x \land A
\land x$ of a skew Boolean algebra are Boolean algebras. For example, the two
identities for the relative complement say that $\setminus$ restricted to a
principal subalgebra acts as the complement operation. See Leech~\cite{L4} for
further details on skew Boolean alebras. We remark that a skew Boolean algebra
with a top element $1$ is degenerate in the sense that it is already a Boolean
algebra. Also note that a skew Boolean algebra whose meet and join are
commutative is the same thing as a generalized Boolean algebra.

Often it is the case that any two elements $x$ and $y$ of a skew Boolean
algebra~$A$ have the greatest lower bound $x \cap y$ with respect to the natural
partial order~$\leq$. When this is the case, we call $\cap$ \emph{intersection}
and speak of a \emph{skew Boolean intersection algebra}. Many examples of skew
lattices occurring in nature posses intersections. The significance of skew
Boolean intersection algebras is witnessed by the fact that they form a
discriminator variety~\cite{BL}, and are therefore both congruence permutable
and congruence distributive. Moreover, results by Bignall and Leech~\cite{BL}
imply that every algebra~$A$ in a pointed discriminator variety is term
equivalent to a right-handed skew Boolean intersection algebra whose congruences
coincide with those of~$A$. In contrast, it was observed already by
Cornish~\cite{Cor} that the congruence lattices of skew Boolean algebras in
general satisfy no particular lattice identity.

Henceforth we shall consider exclusively skew Boolean intersection algebras, so
we simply refer to them as \emph{skew algebras}. A homomorphism of skew algebras
preserves all the operations, namely $0$, $\land$, $\lor$, $\setminus$, and
$\cap$.

Recall that an \emph{ideal}, which we sometimes call \emph{$\preceq$-ideal}, in
a skew algebra $A$ is a subset $I \subseteq A$ which is lower with respect to
$\preceq$ and is closed under finite joins, so in particular $0 \in I$.
An ideal $P$ is \emph{prime} if it is non-trivial and $a \land b \in P$ implies
$a \in P$ or $b \in P$. It can be shown easily that the prime ideals in~$A$
coincide with non-zero maps $A \to 2$ into the two-elments lattice $2 =
\set{0,1}$ which preserve $0$, $\land$, and $\lor$ (but not necessarily~$\cap$).
Because~$2$ is commutative, such maps are in bijective correspondence with
non-zero maps $\D{A} \to 2$. In other words, the assignment
\begin{equation*}
  P \mapsto \D{P} = \set{\dclass{a} \such a \in P}
\end{equation*}
is a bijection from prime ideals in~$A$ to prime ideals in~$\D{A}$. 

We write $f : X \parto Y$ to indicate that $f$ is a partial map from $X$ to $Y$,
defined on its \emph{domain} $\dom{f} \subseteq X$. The \emph{restriction}
$\restricted{f}{D}$ of $f : X \parto Y$ to $D \subseteq X$ is the map~$f$ with
the domain restricted to $\dom{f} \cap D$. We denote the set of all partial maps
from $X$ to $Y$ by $\PP{X}{Y}$. Leech's construction \cite{L4} shows how
$\PP{X}{Y}$ can be endowed with a right-handed skew algbera structure by setting
\begin{align*}
  0 &= \emptyset,\\
  f \land  g &= \restricted{g}{\dom{f} \cap \dom{g}},\\
  f \lor g &= f \cup \restricted{g}{\dom{f} - \dom{g}},\\
  f \setminus g &= \restricted{f}{\dom{f} - \dom{g}},\\
  f \cap g &= f \cap g,
\end{align*}
where $-$ is set-theoretic difference, and the set-theoretic operations on the
right-hand sides act on $f$ and $g$ viewed as functional relations. We
generalize this construction to algebras which are not necessarily right-handed,
because we will need one in Section~\ref{sec:dual-skew-algebr}.

Given a subset $D \subseteq X$, let $\PPd{X}{Y}{D} = \set{f : X \parto Y \mid
  \dom{f} = D}$ be the set of those partial maps $X \parto Y$ whose domain
is~$D$. Suppose we are given for each $D \subseteq X$ a binary operation
$\sand_D$ on $\PPd{X}{Y}{D}$. We say that the family $\set{ {\sand_D} \mid D
  \subseteq X}$ is \emph{coherent} when it commutes with restrictions: for all
$E \subseteq D \subseteq X$ and $f, g \in \PPd{X}{Y}{D}$ it holds
\begin{equation}
  \label{eq:sand_restricted}%
  \restricted{(f \sand_D g)}{E} = \restricted{f}{E} \sand_E \restricted{g}{E}.
\end{equation}
We usually omit the subscript from $\sand_D$ and write just $\sand$.

\begin{theorem}
  \label{theorem:partial-functions-band}
  Let $(\sand_D)_{D \subseteq X}$ be a coherent family of rectangular bands.
  Then $\PP{X}{Y}$ is a skew algebra for the following operations, defined for
  $f, g \in \PP{X}{Y}$ with $\dom{f} = F$ and $\dom{g} = G$:
  \begin{align*}
    0 &= \emptyset,\\
    f \land g &= \restricted{f}{F \cap G} \sand \restricted{g}{F \cap G} ,\\
    f \lor g &=(\restricted{f}{F - G}) \cup (\restricted{g}{G - F}) \cup (g\land f) ,\\
    f \setminus g &= \restricted{f}{F- G},\\
    f \cap g &= f \cap g.
\end{align*}
\end{theorem}

\begin{proof}
  Let $f, g, h \in \PP{X}{Y}$ be partial maps with domains $\dom{f} = F$,
  $\dom{g} = G$, and $\dom{h} = H$, respectively. Observe that the operations
  defined in the statement of the theorem all commute with restrictions, for
  example
  \begin{multline*}
    \restricted{(f \lor g)}{D} =
    \restricted{f}{(F - G) \cap D} \cup \restricted{g}{(G - F) \cap D} \cup
    (\restricted{f}{D \cap F \cap G} \sand \restricted{g}{D \cap F \cap G}) = \\ 
    \restricted{f}{(D \cap F) - (D \cap G)} \cup
    \restricted{g}{(D \cap G) - (D \cap F)} \cup
    (\restricted{f}{D \cap F \cap G} \sand \restricted{g}{D \cap F \cap G}) =
    \restricted{f}{D} \lor \restricted{g}{D}.
  \end{multline*}
  Thus a good strategy for checking an identity is to do it ``by parts'' as
  follows. To check $u = v$ it suffices to check $\restricted{u}{X} =
  \restricted{v}{X}$ and $\restricted{u}{Y} = \restricted{v}{Y}$ separately,
  provided that $\dom{u} = \dom{v} = X \cup Y$. Of course, this only works if
  $X$ and $Y$ are suitably chosen so that the restrictions $\restricted{u}{X}$,
  $\restricted{u}{Y}$, $\restricted{u}{Y}$, and $\restricted{v}{Y}$ simplify
  when we push the restrictions by~$X$ and~$Y$ inwards.

  The following properties are easily verified: $0$ is neutral for $\lor$,
  idempotency of $\land$ and $\lor$, associativity of $\land$. That $\cap$
  computes greatest lower bounds holds because the natural partial order on
  $\PP{X}{Y}$ is subset inclusion $\subseteq$ of functions viewed as functional
  relations. It remains to check associativity of $\lor$, meet distributivity,
  and the properties of $\setminus$.

  Associativity of $\lor$ is checked by parts. The domain of $f \lor (g \lor h)$
  and $(f \lor g) \lor h$ is $F \cup G \cup H$, which is covered by the parts
  $(F \cup G) - H$, $(F \cup H) - G$, $(G \cup H) - F$, and $F \cap G \cap H$.
  On the first part we get
  \begin{equation*}
    \restricted{((f \lor g) \lor h)}{(F \cup G) - H} =
    (\restricted{f}{F - H} \lor \restricted{g}{G - H}) \lor
    \restricted{h}{\emptyset} =
    \restricted{f}{F - H} \lor \restricted{g}{G - H}.
  \end{equation*}
  and
  \begin{equation*}
    \restricted{(f \lor (g \lor h))}{(F \cup G) - H} =
    \restricted{f}{F - H} \lor (\restricted{g}{G - H} \lor
    \restricted{h}{\emptyset}) =
    \restricted{f}{F - H} \lor \restricted{g}{G - H}.
  \end{equation*}
  On $(F \cup H) - G$ and $(G \cup H) - F$ the calculation is similar, while on
  $F \cap G \cap H$ both meets and joins turn into $\sand$ and the identity
  follows as well.

  Next we check that meet distributivity holds. The domain of $f \land (g \lor
  h)$ and $(f \land g) \lor (f \land h)$ is $F \cap (G \cup H)$. It is covered
  by the parts $(F \cap G) - H$, $(F \cap H) - G$, and $F \cap G \cap H$. On the
  first part we get
  \begin{multline*}
    \restricted{(f \land (g \lor h))}{(F \cap G) - H} =
    \restricted{f}{(F \cap G) - H} \land (\restricted{g}{(F \cap G) - H} \lor \restricted{h}{\emptyset}) = \\
    \restricted{f}{(F \cap G) - H} \land \restricted{g}{(F \cap G) - H}
  \end{multline*}
  and
  \begin{multline*}
    \restricted{(f \land g) \lor (f \land h))}{(F \cap G) - H} = \\
    (\restricted{f}{(F \cap G) - H} \land \restricted{g}{(F \cap G) - H}) \lor
    (\restricted{f}{(F \cap G) - H} \land \restricted{h}{\emptyset}) = \\
    (\restricted{f}{(F \cap G) - H} \land \restricted{g}{(F \cap G) - H}) \lor 0 =
    \restricted{f}{(F \cap G) - H} \land \restricted{g}{(F \cap G) - H}.
  \end{multline*}
  The calculation on $(F \cap H) - G$ is similar, and on $F \cap G \cap H$ it
  again trivializes because all operations become $\sand$. For the other half of
  meet distributivity, namely $(f \lor g) \land h = (f \land h) \lor (g \land
  h)$ we use the parts $(F \cap H) - G$, $(G \cap H) - F$, and $F \cap G \cap
  H$.

  Finaly, $\setminus$ satisfies the axioms of relative complementation because
  \begin{equation*}
    (f \setminus g) \land (f \land g \land f) =
    \restricted{f}{F-G} \land \restricted{f}{F\cap G} =
    \emptyset = 0
  \end{equation*}
  and
  \begin{equation*}
    (f \setminus g) \lor (f \land g \land f) =
    \restricted{f}{F-G} \lor \restricted{f}{F\cap G} =
    \restricted{f}{F-G} \cup \restricted{f}{F\cap G} \cup \emptyset =
    f.
  \end{equation*}
\end{proof}

\subsection{Boolean spaces and étale maps}
\label{sec:boolean-spaces-etale-maps}

We start by recaling several standard topological notions. A space is
\emph{zero-dimensional} if its clopens (sets which are both open and
closed) form a topological base. A \emph{Stone space} is a compact
zero-dimensional Hausdorff space, while a \emph{Boolean space} is a
locally compact zero-dimensional Hausdorff space. We call a set which
is compact and open a \emph{copen}. In a Boolean space the copens form
toplogical base.

A partial map $f : X \parto Y$ is said to be continuous when it is continuous as
a map defined on the subset $\dom{f} \subseteq X$ with the induced topology.
Unless noted otherwise, the domain of definition $\dom{f}$ is always going to be
an open subset of~$X$. A continuous map is \emph{proper} if its inverse image
map takes compact subsets to compact subsets, while a partial continuous map
with an open domain of definition is proper when the inverse image $f^{-1}(K)$
of a compact subset $K \subseteq Y$ is compact in $\dom{f}$, or equivalently
in~$X$.

An \emph{étale map $p : E \to B$}, also known as \emph{local
  homeomorphism}, is a continuous map for which~$E$ has an open cover
such that for each $U$ in the cover the restriction $\restricted{p}{U}
: U \to p(U)$ is a homeomorphism onto the image, and $p(U)$ is open in
$B$. We call~$E$ the \emph{total space} and~$B$ the \emph{base} of the
étale map~$p$. The \emph{fiber} above $x \in B$ is the subspace $E_x =
\set{y \in E \such p(y) = x}$. A \emph{section} of~$p$ is a continuous
map $s : U \to E$ defined on a subset $U \subseteq B$, usually open,
such that $p \circ s = \id{U}$.

Étale maps with a common base $B$ form a category, even a topos, in
which morphisms are commutative triangles
\begin{equation*}
  \xymatrix{
    {E} \ar[rr]^{f} \ar[rd]_{p}
    &
    &
    {E'} \ar[ld]^{p'}
    \\
    &
    {B}
    &
  }
\end{equation*}
where $f$ is a continuous map. It follows that $f$ is an étale map,
see for example~\cite[II.6]{maclane92:_sheav_geomet_logic}. One
consequence of this is that a section $s : U \to E$ of an étale map $p
: E \to B$ defined on an open subset $U \subseteq B$ is itself an
étale map (because the inclusion $U \hookrightarrow E$ is an étale
map). In particular, the image $s(U)$ is open in~$E$ and (images of)
open sections form a base for~$E$.

A \emph{copen section} of $p : E \to B$ is a section $s : U \to E$
defined on a copen subset $U \subseteq B$. Its image $s(U)$ is not
only open but also compact in~$E$, and the restriction
$\restricted{p}{s(U)}$ is a homeomorphism from~$s(U)$ onto~$U$.
Conversely, if $S \subseteq E$ is copen and $\restricted{p}{S} : S \to
p(S)$ is a homeomorphism onto $p(S)$ then $(\restricted{p}{S})^{-1}:
p(S) \to S$ is a copen section. This is so because étale maps are
open. We therefore have two views of copen sections: as sections
defined on copen subsets, and as those copen subsets of the total
space which cover each point in the base at most once.

In Section~\ref{sec:dual-skew-algebr} we will need to know how to
compute equalizers and coequalizers in the category of étale maps over
a given base.

\begin{proposition}
  \label{prop:etale-equalizers-coequalizers}
  Let $f$ and $g$ be morphisms of étale maps,
  \begin{equation*}
    \xymatrix{
      {E} \ar[rd]_{p}
      \ar@<0.25em>[rr]^{f} 
      \ar@<-0.25em>[rr]_{g}
      &
      &
      {E'} \ar[ld]^{p'}
      \\
      &
      {B}
      &
    }
  \end{equation*}
  The equalizer and coequalizer of~$f$ and~$g$ in the category of
  étale maps with base~$B$ are computed as in the category of
  topological spaces. Moreover, the quotient map from $E'$ to the
  coequalizer is open.
\end{proposition}

\begin{proof}
  In the category of topological spaces the equalizer of $f$ and $g$
  is the subspace
  \begin{equation*}
    I = \set{x \in E \such f(x) = g(x)}
  \end{equation*}
  with the subspace inclusion $i : I \to E$. For this to be an
  equalizer in the category of étale maps, $p \circ i : I \to B$ must
  be étale, which is the case because $I$ is an open subspace of~$E$.
  To see this, consider any $x \in I$. Because~$f$ and~$g$ are étale
  maps there is an open section $U \subseteq E$ containing $x$ such
  that $\restricted{f}{U} : U \to f(U)$ and $\restricted{g}{U} : U \to
  g(U)$ are homeomorphisms onto open subsets of $E'$. Thus the
  intersection $f(U) \cap g(U)$ is open, from which it follows that $U
  \cap f^{-1}(g(U)) \cap g^{-1}(f(U))$ is an open neighborhood of $x$
  contained in~$I$.

  The coequalizer of $f$ and $g$ is computed in topological spaces as
  the quotient space $Q = E'/R$ where $R \subseteq E' \times E'$ is
  the least equivalence relation generated by the relation $S =
  \set{(f(x), g(x)) \in E' \times E' \such x \in E}$. Because $S$ is
  an open subset of the fibered product $E' \times_B E'$ so is $R$,
  from which it follows that the canonical quotient map~$q : E' \to Q$
  is open. Furthermore, because $p' \circ f = p = p' \circ g$ the map
  $p'$ factors through~$q$ as $p' = r \circ q$. To complete the proof,
  we need to show that~$r$ is étale, but this is easy because we
  already know that~$q$ is open.
\end{proof}

\section{Duality for commutative algebras}
\label{sec:duality-commutative}

Before embarking on duality for skew algebras we review the familiar commutative
case. The \emph{spectrum} $\St{A}$ of a Boolean algebra $A$ is the Stone space
whose points are the prime ideals of~$A$. The elements $a \in A$ correspond to
the basic clopen sets $\baseN{a} = \set{P \in \St{A} \such a \not\in P}$. A
homomorphism $f : A \to A'$ between Boolean algebras induces a continuous map
$\down{f} : \St{A'} \to \St{A}$ that maps a prime ideal~$P$ to its preimage
$\down{f}(P) = f^{-1}(P)$. In the other direction the duality maps a Stone space
to the Boolean algebra of its clopen subsets, with the expected operations of
intersection and union.

A short path to Stone duality for generalized Boolean algebras goes through the
observation that the category $\GBA$ of generalized Boolean algebras is
equivalent to the slice category $\BA/\two$ of Boolean algebras over the initial
algebra~$\two$. By Stone duality for Boolean algebras $\GBA$ is then
dual to \emph{pointed} Stone spaces and continuous maps
which preserve the chosen point. For our purposes it is more convenient to take
yet another equivalent category, namely Stone spaces \emph{without} one point,
which are precisely the Boolean spaces, and suitable partial maps between them.

Let us describe the duality explicitly. Starting from a generalized Boolean
algebra~$A$, we construct its \emph{spectrum} $\St{A}$ as the Boolean space of
prime ideals. An element $a \in A$ corresponds to the basic copen set $\baseN{a}
= \set{P \in \St{A} \such a \not\in P}$.
A homomorphism $f : A \to A'$ between generalized Boolean algebras
induces a partial map $\down{f} : \St{A'} \parto \St{A}$ that maps a
prime ideal~$P$ to its preimage $\down{f}(P)$, provided the preimage
is not all of~$A$. The domain of definition of $\down{f}$ is open, for
if $a \not\in \down{f}(P)$ then $\down{f}$ is defined on
$\baseN{f(a)}$. The fact that the inverse image
$\down{f}^{-1}(\baseN{a})$ of a basic copen is the basic copen
$\baseN{f(a)}$ implies that $\down{f}$ is both continuous and proper.
Indeed, $f$ is proper because the inverse image $f^{-1}(K)$ of a
compact subset $K \subseteq \St{A}$ is a closed subset of the
compact set $f^{-1}(\baseN{a_1}) \cup \cdots \cup
f^{-1}(\baseN{a_n})$ where $\baseN{a_1}, \ldots, \baseN{a_n}$ is some
finite cover of $K$ by basic copen sets.
In summary, the category of generalized Boolean algebras is equivalent to the
category of Boolean spaces and proper continuous partial maps with open domains
of definition.

\section{Duality for skew algebras}
\label{sec:duality-skew}

Any attempt at extension of Stone duality naturally leads to consideration of
prime ideals. Since the one-point space~$\one$ corresponds to the initial
Boolean algebra~$\two$, a point $\one \to X$ on the topological side corresponds
to homomorphisms $A \to \two$ on the algebraic side. However, since such
homomorphisms factor through the commutative reflection $A \to \D{A}$,
they give us insufficient information about the non-commutative structure
of~$A$. We should therefore expect that on the topological side we have to look
for structures that can be rich even though they have few \emph{global} points,
while on the algebraic side we cannot afford to use ideals exclusively, but must also
consider congruence relations.

In the case of a commutative algebra the congruence relation generated by a
prime ideal has just two equivalence classes. In contrast, the least congruence
relation $\theta_P$ whose zero-class contains the prime ideal $P \subseteq A$ in
a skew algebra~$A$ generally has many equivalence classes, which ought to be
accounted for on the topological side of duality. The following result of
Bignall and Leech~\cite{BL} characterizes the congruence relations generated by
ideals in a skew algebra.

\begin{lemma}[Bignall \& Leech]
  \label{lemma:cong}
  For an ideal $I \subseteq A$ in a skew algebra~$A$ let $\thetaeq_I$ be the
  least congruence relation on~$A$ whose zero-class contains~$I$. Then for all
  $x, y \in A$, $x \thetaeq_I y$ if, and only if, $(x \setminus (x \cap y)) \lor
  (y \setminus (x \cap y)) \in I$.
\end{lemma}

\noindent
In fact, the zero-class of $\thetaeq_I$ equals $I$. Consequently, the ideals of
a skew algebra are in bijective correspondence with congruence relations. A
consequence of Lemma~\ref{lemma:cong} is that $x \thetaeq_{f^{-1}(I)} y$ is
equivalent to $f(x) \thetaeq_I f(y)$, for any skew algebra homomorphism $f : A
\to A'$ and any ideal $I \subseteq A'$. From this we obtain the following
properties of prime ideals.

\begin{lemma}
  \label{lemma:element-max-ideal}
  \label{lemma:s-less-than-t-not-in-P}
  \label{lemma:int-in-P}
  Let $P$ be a prime ideal in a skew algebra $A$ and let $x, y \in A$. Then:
  \begin{enumerate}
  \item  $x \in P$ or $y \setminus x\in P$.
  \item If $x \leq y$ and $x \not\in P$ then $x \thetaeq_P y$.
  \item If $x\leq y$ then $x \in P$ is equivalent to $y \thetaeq_P y \setminus x$.
  \end{enumerate}
\end{lemma}

\begin{proof}
  \begin{enumerate}
  \item We have $x \land (y \setminus x) = 0\in P$. Because $P$ is a prime
    ideal it follows that $x\in P$ or $y\setminus x\in P$.
  \item We need to show that $(x \setminus (x\cap y)) \lor (y\setminus (x\cap
    y)) \in P$. Since $x \leq y$ it follows that $x \cap y = x$ and thus we only
    need $y \setminus x \in P$, which follows from the first statement of the lemma.
  \item The element $x'=y\setminus x$ is the complement of $x$
    in the Boolean algebra $y \land A \land y$. Hence $x'\leq y$ and $y\setminus x' = x$. 
   Thus Lemma \ref{lemma:cong} implies that
   $y\thetaeq_P x'$ is equivalent to $(y\setminus x') \lor (x'\setminus x') = x \in P$.
  \end{enumerate}
\end{proof}

\subsection{From algebras to spaces}
\label{sec:algebras-to-spaces}

Following our own advice that the topological side of duality should account for
the equivalence classes of congruences, we define the
\emph{skew spectrum} $\Sk{A}$ of a skew algebra $A$ to be the space whose points
are pairs $(P, e)$, where $P$ is a prime ideal in~$A$ and $e$ is a non-zero
equivalence class of~$\theta_P$. Since every non-zero equivalence class equals
$[t]_{\theta_P}$ for some $t \not\in P$, a general element of the skew spectrum
may be written as $(P, [t]_{\theta_P})$. We write just $\elemsk{P}{t}$. Thus, as
a set the skew spectrum is
\begin{equation*}
  \Sk{A} = \set{
    \elemsk{P}{t} \such
    \text{$P$ prime ideal in $A$ and $t \not\in P$}
  },
\end{equation*}
where $\elemsk{P}{t} = \elemsk{Q}{u}$ when $P = Q$ and $t \thetaeq_P u$. The
topology of $\Sk{A}$ is the one whose basic open sets are of the form, for $a
\in A$,
\begin{equation*}
  \baseM{a} = \set{\elemsk{P}{a} \such a\not\in P}.
\end{equation*}
These really form a basis because they are closed under intersections.

\begin{lemma}
  \label{lemma:baseM-preserves-cap}
  Let $a,b\in A$. Then  $\baseM{a} \cap \baseM{b} = \baseM{a \cap b}$.
\end{lemma}

\begin{proof}
 For a prime ideal $P$ and $t \in A$ the statement $\elemsk{P}{t}
  \in \baseM{a} \cap \baseM{b}$ amounts to
  \begin{equation}
    \label{eq:baseM-preserve-cap-1}
    (a \not\in P) \land (b \not\in P) \land (t \thetaeq_P a) \land
    (t \thetaeq_P b),
  \end{equation}
  while $\elemsk{P}{t} \in \baseM{a \cap b}$ means
  \begin{equation}
    \label{eq:baseM-preserve-cap-2}
    (a \cap b \not\in P) \land (t \thetaeq_P a \cap b).
  \end{equation}
  If~\eqref{eq:baseM-preserve-cap-1} holds then $a\cap b \thetaeq_P a \cap a = a
  \thetaeq_P t$ which proves~\eqref{eq:baseM-preserve-cap-2}.

  To prove the converse, suppose~\eqref{eq:baseM-preserve-cap-2} holds. Because
  $a \cap b \not\in P$, $a \cap b \leq a$ and $a \cap b \leq b$, Lemma
  \ref{lemma:s-less-than-t-not-in-P} implies $a \thetaeq_P a \cap b \thetaeq_P b$
  which suffices for~\eqref{eq:baseM-preserve-cap-1}.
\end{proof}

The skew spectrum on its own contains too little information to act as the dual.
For example, both $\Sk{\two \times \two}$ and $\Sk{\three}$ are the discrete
space on two points. Recall that $\three$ is the right-handed skew algebra whose
set of elements is $\set{0,1,2}$ and whose $\DD$-classes are $\set{0}$ and
$\set{1,2}$. One part of the missing information is provided by the map $q_A :
\Sk{A} \to \St{A}$ defined by
\begin{equation*}
  q_A(\elemsk{P}{t}) = \D{P},
\end{equation*}
where we used the shorthand $\St{A} = \St{\D{A}}$.

\begin{proposition}
  The map $q_A : \Sk{A} \to \St{A}$ is onto and étale.
\end{proposition}

\begin{proof}
  Because prime ideals are non-trivial the map $q_A$ is onto. To show that
  $q_A$ is continuous, we prove
  \begin{equation}
    \label{eq:inverse-image-baseN-open}
    q_A^{-1}(\baseN{a})=\bigcup_{b \preceq a} \baseM{b},
  \end{equation}
  where we used the shorthand $\baseN{a} = \baseN{\dclass{a}}$. For one
  inclusion, observe that $b \preceq a$ implies $\dclass{b} \leq \dclass{a}$ and
  hence $q_A(\baseM{b}) = \baseN{b} \subseteq \baseN{a}$. For the other inclusion,
  suppose $q_A(\elemsk{P}{t}) \in \baseN{a}$. Then $t \land a \land t \not\in P$
  because $a \not\in P$ and $t \not\in P$, and by
  Lemma~\ref{lemma:s-less-than-t-not-in-P} we get $t \thetaeq_P t \land a \land
  t$ from which $\elemsk{P}{t} \in \baseM{t \land a \land t}$ follows.

  Finally, $q_A$ is étale because its restriction to a basic open set
  $\baseM{a}$ is a (continuous) bijection onto the basic open set $\baseN{a}$.
\end{proof}

\begin{corollary}
  The skew spectrum $\Sk{A}$ is a Boolean space.
\end{corollary}

\begin{proof}
  Local compactness and zero-dimensionality are lifted from $\St{A}$
  to $\Sk{A}$ by the étale map $q_A$. To see that $\Sk{A}$ is
  Hausdorff, let $\elemsk{P}{s}$ and $\elemsk{Q}{t}$ in $\Sk{A}$ be
  two distinct points in $\Sk{A}$.

  Consider the case $P = Q$. We claim that $\baseM{s \setminus (s \cap t)}$ and
  $\baseM{t \setminus (s \cap t)}$ are disjoint and are neighborhoods of $s$ and
  $t$, respectively. Disjointness follows by
  Lemma~\ref{lemma:baseM-preserves-cap} from the fact that $(t \setminus (s \cap
  t)) \cap (s \setminus (s \cap t)) = 0$. From $\lnot (t \thetaeq_P s)$ we
  conclude by Lemma \ref{lemma:cong} that $(s \setminus (s \cap t)) \lor (t
  \setminus (s \cap t)) \not\in P$, thus $s \setminus (s \cap t) \not\in P$ or
  $t \setminus (s \cap t) \not\in P$. In either case $s \cap t\in P$ by (1) of Lemma
  \ref{lemma:element-max-ideal}, therefore $s \thetaeq_P s \setminus (s\cap t)$
  by (3) of Lemma \ref{lemma:int-in-P}, which proves $\elemsk{P}{s}\in \baseM{s
    \setminus (s\cap t)}$, as claimed. We similarly show that $\elemsk{P}{t} \in
  \baseM{t \setminus (s \cap t)}$.

  Consider the case $P \neq Q$. Because $\St{A}$ is Hausdorff there exist
  disjoint basic open nehigborhoods $\baseN{a}$ and $\baseN{b}$ of $q_A(P)$ and
  $q_A(Q)$, respectively. Thus $q_A^{-1}(\baseN{a})$ and $q_A^{-1}(\baseN{b})$
  are disjoint and are open neighborhoods of $\elemsk{P}{s}$ and
  $\elemsk{Q}{t}$, respectively.
\end{proof}

Because Boolean spaces are zero-dimensional, the étale map $q_A : \Sk{A} \to
\St{A}$ has more sections that one would normally expect.

\begin{proposition}
  \label{prop:copens-have-sections}
  The étale map $q_A : \Sk{A} \to \St{A}$ has a section above every
  copen set, passing through a prescribed point above the copen set.
\end{proposition}

\begin{proof}
  More precisely, the proposition claims that given a copen set $U \subseteq
  \St{A}$ and a point $x \in \Sk{A}$ such that $q_A(x) \in U$, there is a copen
  section above~$U$ containing~$x$. For the proof we only need surjectivity of
  $q_A$ and zero-dimensionality of $\St{A}$. By surjectivity the copen set $U$
  can be decomposed into pairwise disjoint copen sets $U_1, \ldots, U_n$, each
  of which has a section~$s_i$. The sections can be glued together into a
  section~$s$ above~$U$. To make sure that $s$ passes through a prescribed
  point~$x$, just change it suitably on a small copen neighborhood of $q_A(x)$.
\end{proof}

\begin{proposition}
  \label{prop:baseM-all-copen-sections}
  The copen sections of $q_A : \Sk{A} \to \St{A}$ are precisely the sets
  $\baseM{a}$ with $a \in A$.
\end{proposition}

\begin{proof}
  We already know that each $\baseM{a}$ is a copen section because $q_A$  maps it  homeomorphically onto the copen set $\baseN{a}$. Conversely, let $V$
  be a copen section in $\Sk{A}$. It is a finite union of basic copen sets $V =
  \baseM{a_1} \cup \cdots \cup \baseM{a_n}$. We show that $V = \baseM{a_1 \lor \cdots
    \lor a_n}$. If $n = 1$ there is nothing to prove.

  Consider the case $n = 2$. The fact that $V = \baseM{a_1} \cup \baseM{a_2}$ is
  a section amounts to: for all prime ideals $P$ in $A$, if $a_1 \not\in P$ and
  $a_2 \not\in P$ then $a \thetaeq_P b$. Let us show that this implies
  $\baseM{a_1} \subseteq \baseM{a_1 \lor a_2}$. If $a_1 \not\in P$ then $a_1
  \lor a_2 \not\in P$. Either $a_2 \in P$ or $a_2 \not\in P$. In the first case
  $a_2 \thetaeq_P 0$ and so $a_1 \thetaeq_P a_1 \lor 0 \thetaeq_P a_1 \lor a_2$.
  In the second case, $a_2 \thetaeq_P a_1$ and so $a_1 \thetaeq_P a_1 \lor a_1
  \thetaeq_P a_1 \lor a_2$. We similarly show that $\baseM{a_2} \subseteq
  \baseM{a_1 \lor a_2}$, from which we get $\baseM{a_1} \cup \baseM{a_2}
  \subseteq \baseM{a_1 \lor a_2}$.
 
 To complete the case $n = 2$ we still have to prove $\baseM{a_1 \lor a_2}
 \subseteq \baseM{a_1} \cup \baseM{a_2}$. Suppose $a_1\lor a_2 \not\in P$. Then
 either $a_1\not\in P$ or $a_2 \not\in P$. We consider the case $a_1\not\in P$,
 the other one is similar. Either $a_2 \in P$ or $a_2 \not\in P$. In the first
 case $a_2 \thetaeq_P 0$ so $a_1 \thetaeq_P a_1 \lor 0 \thetaeq_P a_1 \lor a_2$.
 In the second case the assumption gives us $a_1 \thetaeq_P a_2$ and so
 $a_1\thetaeq_P a_1 \lor a_1 \thetaeq_P a_1 \lor a_2$.

 The cases $n > 2$ follow by repeated use of the case $n = 2$.
\end{proof}

Later on we will need to know how the map $a \mapsto \baseM{a}$ interacts with
the operations of~$A$. For this purpose we define the \emph{saturation}
operation $\sigma_A(U) = q_A^{-1}(q_A(U))$. Since $q_A$ is open, the
saturation of a copen is a clopen.

\begin{proposition}
  \label{prop:baseM-properties}
  The assignment $a \mapsto \baseM{a}$ is an order-isomorphism from a skew
  algebra~$A$, ordered by the natural partial order, onto the copen sections of
  $\Sk{A}$, ordered by subset inclusion, satisfying the identities
  \begin{equation*}
  \baseM{a \land b \land a} = \sigma_A(\baseM{b}) \cap \baseM{a}
  \quad\text{and}\quad
  \baseM{a \lor b \lor a} = \baseM{a} \cup (\baseM{b} \setminus \sigma_A(\baseM{a})).
  \end{equation*}
\end{proposition}

\begin{proof}
 %
  If $\baseM{a} = \baseM{b}$ then $a \thetaeq_P b$ for all prime ideals $P$.
  Thus for any prime ideal $P$
  \begin{equation*}
    (a \setminus (a \cap b)) \lor (b \setminus (a \cap b)) \in P,
  \end{equation*}
  and so both $a \setminus (a \cap b)$ and $b \setminus (a \cap b)$ belong to $P$.
  Because the intersection of all prime ideals is $\set{0}$, it follows that
  \begin{equation*}
    a \setminus (a \cap b) = 0 = b \setminus (a \cap b).
  \end{equation*}
  Therefore, $\dclass{a} = \dclass{a \cap b} = \dclass{b}$, which is only
  possible if $a = b$. It follows that the assignment $a \mapsto \baseM{a}$ is
  injective, while its surjectivity follows from
  Proposition~\ref{prop:baseM-all-copen-sections}.

  That $a \mapsto \baseM{a}$ is an order isomorphism follows from the following
  chain of equivalences, where we use Lemma~\ref{lemma:baseM-preserves-cap} and
  injectivity of $a \mapsto \baseM{a}$ in the second and third step,
  respectively:
  \begin{multline*}
    a \leq b
    \iff a \cap b = a
    \iff \baseM{a \cap b} = \baseM{a} \\
    \iff \baseM{a} \cap \baseM{b} = \baseM{a}
    \iff \baseM{a} \subseteq \baseM{b}.
  \end{multline*}
  It remains to prove the two identites. For the first one, let
  $\elemsk{P}{a\land b\land a}\in \baseM{a\land b\land a}$. Then $a\land b\land
  a\not\in P$ which implies both $b\not\in P$ (and thus $\elemsk{P}{a\land
    b\land a}\in \sigma_A (\baseM{b})$) and $a\not\in P$. We have $a\land b\land
  a\leq a$ and $a\land b\land a\thetaeq_P a$ follows by Lemma
  \ref{lemma:s-less-than-t-not-in-P}. Hence $\elemsk{P}{a\land b\land a}\in
  \sigma_A (\baseM{b}) \cap \baseM{a}$.
 
 To prove the converse, let $\elemsk{P}{a}\in \sigma_A (\baseM{b})\cap \baseM{a}$. Hence $a\not\in P$
 and $b\not\in P$. So $a\land b\land a\not\in P$, $a\land b\land a\leq a$ and thus 
 $a\land b\land a\thetaeq_P a$
 by Lemma \ref{lemma:s-less-than-t-not-in-P}.
 
 To prove the second equality first assume that 
 $\elemsk{P}{a\lor b\lor a}\in \baseM{a\lor b\lor a}$. Hence $a\lor b\lor a\not\in P$. 
 If $a\not\in P$ then $a\thetaeq_P  a\lor b\lor a$ by Lemma  \ref{lemma:s-less-than-t-not-in-P}
 because always $a\leq a\lor b\lor a$, and $\elemsk{P}{a\lor b\lor a} \in \baseM{a} $
 follows. If $a\in P$ then $b\not\in P$ follows from $a\lor b\lor a\not\in P$. Thus  
 $a\lor b\lor a \thetaeq_P 0\lor b\lor 0\thetaeq_P b$ 
 and $\elemsk{P}{a\lor b\lor a}\in \baseM{b}\backslash \sigma_A (\baseM{a})$ follows.
 
 Finally, assume that $\elemsk{P}{s}\in \baseM{a}\cup (\baseM{b}\backslash \sigma_A (\baseM{a}))$. Then
 either $a\not\in P$ and $\elemsk{P}{s}=\elemsk{P}{a}$,  or $a\in P$,
 $b\not\in P$ and $\elemsk{P}{s}=\elemsk{P}{b}$. In the first case it follows
 that $a\lor b\lor a\not\in P$ and thus $a\thetaeq_P a\lor b\lor a$ by Lemma  
 \ref{lemma:s-less-than-t-not-in-P}. Therefore $\elemsk{P}{s}=\elemsk{P}{a}\in \baseM{a\lor b\lor a}$.
 In the second case it follows that $a\lor b\lor a\not\in P$ and
 $a\lor b\lor a\thetaeq_P 0\lor b\lor 0\thetaeq_P  b$. Again, 
$\elemsk{P}{s}=\elemsk{P}{b}\in \baseM{a\lor b\lor a}$.
\end{proof}

The attentive reader will point out that the étale map $q_A : \Sk{A} \to \St{A}$
cannot possibly be the topological dual of the skew algebra~$A$ because $A$
gives the same étale map as its opposite algebra (in which the operations
$\land$ and $\lor$ are the mirror version of those in~$A$). Indeed, the étale
map provides sufficient information only when~$A$ is right-handed (or
left-handed), as will be shown in Section~\ref{sec:skew-rh-duality-proofs}. We
shall consider the general case in Section~\ref{sec:dual-skew-algebr}.

Let $f : A \to A'$ be a homomorphism of skew algebras. As we explained in
Section~\ref{sec:duality-commutative}, the induced homomorphism $\D{A} \to
\D{A'}$ is dual to a proper partial map $\down{f} : \St{A'} \to \St{A}$ with an
open domain, which maps prime ideals to their inverse images, when defined.
There is also a map $\up{f} : \Sk{A'} \to \Sk{A}$ of the same kind between the
skew spectra. It is characterized by the requirement
\begin{equation*}
  \up{f} (\elemsk{P}{f(a)}) = \elemsk{\down{f}(P)}{a},
\end{equation*}
which uniquely determines the value of $\up{f}$, when defined, because
$a \thetaeq_{\down{f}(P)} b$ is equivalent to $f(a) \thetaeq_P f(b)$.
If $\up{f}(\elemsk{P}{f(a)})$ is defined then $f(a) \not\in P$, hence
$\up{f}$ is defined on the basic copen set $\baseM{f(a)}$, which shows
that the domain of definition of $\up{f}$ is open. Next, $\up{f}$ is
continuous and proper because $\up{f}^{-1}(\baseM{a}) = \baseM{f(a)}$
for all $a \in A$.

The square
\begin{equation}
  \label{eq:sk-up-down-square}
  \xymatrix@+1em{
    {\Sk{A}} \ar[d]_{q_A}
    &
    {\Sk{A'}}
    \ar@_{->}[l]_{\up{f}}
    \ar[d]^{q_{A'}}
    \\
    {\St{A}}
    &
    {\St{A'}}      
    \ar@_{->}[l]^{\down{f}}
  }
\end{equation}
commutes. It is easy to see that the values of $q_A \circ \up{f}$ and $\down{f}
\circ q_B$ coincide whenever they are both defined, so we only show that they
have the same domain of definition. If $\down{f}(P)$ is defined then there is $a
\in A$ such that $f(a) \not\in P$, in which case $\up{f}$ is defined at
$\elemsk{P}{f(a)}$ and $q_B(\elemsk{P}{f(a)}) = P$. Conversely, if
$\up{f}(\elemsk{P}{f(a)}$ is defined then $\down{f}(q_B(\elemsk{P}{f(a)})) =
\down{f}(P)$ is defined because $f(a) \not\in P$.

\begin{lemma}
  \label{lemma:dagger-bijective-on-fibers}
  The map $\up{f}$ is a bijection on fibers, i.e., given any $P \in
  \dom{\down{f}}$, $\up{f}$ maps $\Sk{B}_P \cap \dom{\up{f}}$
  bijectively onto the fiber $\Sk{A}_{\down{f}(P)}$.
\end{lemma}

\begin{proof}
  Consider any $P \in \dom{\down{f}}$. The Lemma states that $\up{f}$
  is a bijective map between the sets $\set{\elemsk{P}{f(a)} \such
    f(a) \not\in P}$ and $\set{\elemsk{\down{f}(P)}{a} \such a \not\in
    \down{f}(P)}$.

  For injectivity, suppose $\up{f}(\elemsk{P}{f(a)}) =
  \up{f}(\elemsk{P}{f(a')})$ where $f(a) \not\in P$ and $f(a') \not\in P$.
  By the definition of $\up{f}$ it follows that
  $\elemsk{\down{f}(P)}{a} = \elemsk{\down{f}(P)}{a'}$, which is equivalent to
  $a \thetaeq_{\down{f}(P)} a'$, and hence $f(a) \thetaeq_P f(a')$. This
  establishes the fact that $\elemsk{P}{f(a)} = \elemsk{P}{f(a')}$.

  For surjectivity, pick any $\elemsk{\down{f}(P)}{a}$ where $a\not\in
  \down{f}(P)$. It follows that $f(a) \not\in P$, so $\up{f}$ is
  defined at $\elemsk{P}{f(a)}$ and maps it to $\elemsk{\down{f}(P)}{a}$.
\end{proof}

\noindent
In view of the previous lemma one might contemplate turning $\up{f}$ around, so
that instead of a partial map which is bijective on fibers we would get a total
map which is injective on fibers. The trouble is that the inverted map need not
be continuous, so the topological nature of $\up{f}$ must be obscured by a more
complex condition.

\subsection{From spaces to algebras}
\label{sec:spaces-to-algebras}

As we already indicated, the original algebra~$A$ cannot always be
reconstructed from $q_A : \Sk{A} \to \St{A}$. However, if $A$ is
\emph{right-handed} then $q_A : \Sk{A} \to \St{A}$ carries all the
information needed, so we consider this case first.

We call a surjective étale map $p : E \to B$ between Boolean spaces a
\emph{skew Boolean space}. The corresponding right-handed skew algebra
$A_p$ consists of copen sections of~$p$, i.e., an element of $A_p$ is
a copen subset $S \subseteq E$ such that $\restricted{p}{S}$ is
injective. To describe the right-handed skew structure on $A_p$,
recall the saturation operation $\sigma_p(S) = p^{-1}(p(S))$, and
define for $S, R \in A_p$
\begin{align*}
  0 &= \emptyset,\\
  S \land R &= \sigma_p(S) \cap R, \\
  S \lor R  &=  S \cup (R - \sigma_p(S)), \\
  S \setminus R &= S - \sigma_p(R), \\
  S \cap R &= S \cap R.
\end{align*}
It is clear that these operations map back into $A_p$. For example, $S \cap R$
is a section, and it is copen because it is the intersection of two copen
subsets of the Hausdorff space~$E$.

It is not difficult to check that the above operations form a
skew algebra with bare hands. An alternative, more elegant way of establishing the
skew structure is to view the elements of $A_p$ as
partial maps $s, r: B \parto E$ and exhibit $A_p$ as a subalgebra of the 
right-handed skew algebra $\PP{B}{E}$ as described in Section~\ref{sec:skew-boolean-algebras}.

The construction of the skew algebra $A_p$ induces the usual
construction of a generalized Boolean algebra via the lattice
reflection, as follows.

\begin{proposition}
  \label{prop:p-reflection}
  Let $p : E \to B$ be a skew Boolean space and let $B^{*}$ be the
  Boolean algebra of copen subsets of~$B$. The map $A_p \to B^{*}$
  defined by $S \mapsto p(S)$ is the lattice reflection of~$A_p$.
\end{proposition}

\begin{proof}
  By Proposition~\ref{prop:copens-have-sections} the map~$S \mapsto
  p(S)$ is surjective, and it is easily seen to be a lattice
  homomorphism. Thus we only have to check $p(S) = p(R)$ is equivalent
  to $S \mathbin{\DD} R$, where $S$ and $R$ are copen sections of~$p$.
  This follows from
  \begin{multline*}
    S \preceq R
    \iff S \land R \land S = S
    \iff \sigma_p(S) \cap \sigma_p(R) \cap S = S \\
    \iff S \subseteq \sigma_p(R)
    \iff p(S) \subseteq p(R).
  \end{multline*}
\end{proof}

We turn attention to morphisms next. We already know that a
homomorphism $f : A \to A'$ of skew algebras induces a commutative
square~\eqref{eq:sk-up-down-square} in which the horizontal partial
maps are proper, continuous and have open domains of definition, and
that the top map is bijective on fibers by
Lemma~\ref{lemma:dagger-bijective-on-fibers}. So we define a morphism
$(g,h)$ between skew Boolean spaces $p : E \to B$
and $p' : E' \to B'$ to be a commutative diagram
\begin{equation}
  \label{eq:catS-morphism}
  \xymatrix@+1em{
    {E}
    \ar@^{->}[r]^{g}
    \ar[d]_{p}
    &
    {E'}
    \ar[d]^{p'}
    \\
    {B}
    \ar@^{->}[r]_{h}
    &
    {B'}
  }
\end{equation}
in which $g$ and $h$ are proper continuous partial maps with open
domains of definition. Furthermore, we require that $g$ is a bijection
on fibers, in the sense that $g$ maps $E_x \cap \dom{g}$ bijectively
onto $E_{h(x)}$ for every $x \in \dom{h}$.

\begin{lemma}
  \label{lemma:morphism-reflection}
  If $(g,h)$ is a morphism between skew Boolean spaces $p : E \to B$
  and $p' : E' \to B'$ then $p(g^{-1}(S)) = h^{-1}(p'(S))$ for every
  copen section $S$ in $E'$.
\end{lemma}

\begin{proof}
  Let $S$ be a copen section in~$E'$. If $y \in p(g^{-1}(S))$ then there exists
  $x \in g^{-1}(S)$ such that $y = p(x)$. Then $h(y) = h(p(x)) = p'(g(x)) \in
  p'(S)$, so that $y \in h^{-1}(p'(S))$ and $p(g^{-1}(S)) \subseteq
  h^{-1}(p'(S))$ follows. On the other hand, if $h(y) \in p'(S)$ then
  $h(y) = p'(z)$ for some $z \in S$. Because $g$ is surjective on fibers there
  exists $x \in E_y$ with the property $g(x) = z$. Now, $y = p(x)\in
  p(g^{-1}(S))$ and we get $ h^{-1}(p'(S))\subseteq p(g^{-1}(S))$.
\end{proof}

We would like to construct a corresponding homomorphism $\homo{(g,h)} : A_{p'}
\to A_p$. For a copen section $S \subseteq E'$, the inverse image
$g^{-1}(S)$ is copen in $E$ because $g$ is continuous and proper, and
it is a section because~$g$ is injective on fibers. Thus we may define
$\homo{(g,h)}(S) = g^{-1}(S)$.

Let us show that $g^{-1}$ commutes with the saturation operations
$\sigma_p$ and $\sigma_{p'}$. If $S$ is a copen section in $E'$, then
$p(g^{-1}(S)) = h^{-1}(p'(S))$ by
Lemma~\ref{lemma:morphism-reflection}. Therefore
\begin{equation*}
  g^{-1}(\sigma_{p'}(S)) =
  p^{-1}(h^{-1}(p'(S)) =
  p^{-1}(p(g^{-1}(S)) =
  \sigma_p(g^{-1}(S)),
\end{equation*}
as claimed. The operations on $A_p$ and $A_{p'}$ are defined in terms of basic
set-theoretic operations and the saturation maps $\sigma_p$ and $\sigma_p'$.
Because $g^{-1}$ commutes with all of them, $\homo{(g,h)}$ is an algebra
homomorphism.

The following is the counterpart of Proposition~\ref{prop:p-reflection} for
morphisms.

\begin{proposition}
  \label{p-reflection-morphisms}
  Let $(g,h)$ be a morphism between skew Boolean spaces $p : E \to B$
  and $p' : E' \to B'$. Its lattice reflection $\D{\homo{(g,h)}} :
  \D{A_{p'}} \to \D{A_p}$ is isomorphic to the lattice homomorphism
  $h^{*} : (B')^{*} \to B^{*}$ defined by $h^{*}(S) = h^{-1}(S)$.
\end{proposition}

\begin{proof}
  By Proposition~\ref{prop:p-reflection} the vertical maps in the
  diagram
  \begin{equation*}
    \xymatrix@+1em{
      {A_p} \ar[d]
      &
      {A_{p'}} \ar[l]_{\homo{(g,h)}} \ar[d]
      \\
      {B^{*}}
      &
      {(B')^{*}} \ar[l]^{h^{*}}
    }
  \end{equation*}
  are lattice reflections. Because the right-hand vertical arrow is
  epi it suffices to show that the diagram commutes, which is just
  Lemma~\ref{lemma:morphism-reflection}.
\end{proof}

\subsection{Duality for right-handed algebras}
\label{sec:skew-rh-duality-proofs}

Let us take stock of what we have done so far, in terms of category theory. On
the algebraic side we have the category $\CatA$ of skew algebras and homomorphisms,
as well as its reflective full subcategory $\CatRA$ on right-handed skew
algebras.

On the topological side we have the category $\CatS$ of skew Boolean spaces. A
morphism $(g,h) : p \to p'$ between skew Boolean spaces $p : E \to B$ and $p' :
E' \to B'$ is a commutative square~\eqref{eq:catS-morphism} with proper
continuous partial maps~$g$ and~$h$ whose domains of definition are open, and
the top map~$g$ is bijective on fibers. Admittedly, the morphisms in $\CatS$ are
not very nice, but in Section~\ref{sec:variations} we show that they decompose
nicely into partial identities and pullbacks.

In Section~\ref{sec:algebras-to-spaces} we defined a functor
\begin{equation*}
  \FunS : \op{\CatA} \to \CatS
\end{equation*}
which assigns to each skew algebra $A$ a skew Boolean space $\FunS(A) = q_A :
\Sk{A} \to \St{A}$. The functor takes a homomorphism $f : A \to B$ to the
corresponding morphism $\FunS(f) = (\up{f}, \down{f}) : \FunS(B) \to \FunS(A)$.

In Section~\ref{sec:spaces-to-algebras} we defined a functor
\begin{equation*}
  \FunA : \op{\CatS} \to \CatRA
\end{equation*}
which maps an étale map $p : E \to B$ between Boolean spaces to a right-handed
skew algebra $\FunA(p) = A_p$, and a morphism $(g,h) : p \to p'$ as
in~\eqref{eq:catS-morphism} to a homomorphism $\FunA(g,h) = \homo{(g,h)} :
\FunA(p') \to \FunA(p)$.

We now work towards showing that $\FunS$ restricted to $\CatRA$ and $\FunA$ form
a duality. For a skew algebra $A$ define the map $\phi_A : A \to
\FunA(\FunS(A))$ by
\begin{equation*}
  \phi_A(a) = \baseM{a}.
\end{equation*}
That $\phi_A$ is an isomorphism of right-handed skew algebras follows from
Lemma~\ref{lemma:baseM-preserves-cap} and
Proposition~\ref{prop:baseM-properties}. By the lemma $\phi_A$ preserves $\cap$,
it obviously preserves $0$, and by the proposition it is a bijection which
preserves the right-handed skew operations $\land$ and $\lor$:
\begin{align*}
  \baseM{a \land b} &=
  \baseM{b \land a \land b} =
  \sigma_A(\baseM{a}) \cap \baseM{b} =
  \baseM{a} \land \baseM{b}
  \\
  \baseM{a \lor b} &=
  \baseM{a \lor b \lor a} =
  \baseM{a} \cup (\baseM{b} \setminus \sigma_A(\baseM{a})) =
  \baseM{a} \lor \baseM{b}.
\end{align*}
Naturality of $\phi$ amounts to the identity
\begin{equation*}
  \up{f}^{-1}(\baseM{a}) = \baseM{f(a)},
\end{equation*}
where $f : A \to B$ is a homomorphism and $a \in A$. After unraveling
the definition of $\up{f}$ we see that the set on the left-hand side
consists of elements $\elemsk{P}{f(a)}$ with $f(a) \not\in P$, which
is just the description of the right-hand side. We have shown that
$\phi$ is a natural isomorphism between the identity and $\FunA \circ
\FunS$.

To establish the equivalence we also need an isomorphism $\psi_p$ between a skew
Boolean space $p : E \to B$ and $q_A : \Sk{A_p} \to \St{A_p}$, natural in $p$.
It consists of two homeomorphism $\down{(\psi_p)} = h$ and $\up{(\psi_p)} = g$
for which the following diagram commutes:
\begin{equation}
  \label{eq:psi-square}
  \xymatrix@+1em{
    {E} \ar[r]^(0.4){g} \ar[d]_{p}
    &
    {\Sk{A_p}} \ar[d]^{q_{A_p}}
    \\
    {B} \ar[r]_(0.4){h}
    &
    {\St{A_p}}
  }
\end{equation}
Because the top map determines the bottom one, we consider~$g$ first. By
Proposition~\ref{prop:baseM-properties} the map $S \mapsto \baseM{S}$ is an
order isomorphism between $A_p$ and the copen sections of $\Sk{A_p}$. But since
$A_p$ is the set of copen sections of~$E$, and the natural partial order in
$A_p$ coincides with the subset relation in $E$, the map $S \mapsto \baseM{S}$
maps the basis for~$E$ isomorphically onto the basis for $\Sk{A_p}$.
Consequently, the topologies of $E$ and $\Sk{A_p}$ are isomorphic as posets,
too, and because $E$ and $\Sk{A_p}$ are sober spaces they are homeomorphic.
Explicitly, the homeomorphism $g : E \to \Sk{A_p}$ induced by the isomorphism $S
\mapsto \baseM{S}$ takes a point $y \in E$ to the unique point $g(y) \in
\Sk{A_p}$ satisfying, for all copen sections $S$ in $E$,
\begin{equation*}
  y \in S \iff g(y) \in \baseM{S}.
\end{equation*}
Similarly, the homeomorphism $h : B \to \St{A_p}$ is characterized by the
requirement, for all copen sections $S$ in $E$,
\begin{equation*}
  x \in p(S) \iff h(x) \in \baseN{S}.
\end{equation*}
It is not hard to verify that $h(x) = \set{\dclass{R} \such R \in A_p \land x
  \in p(R)}$ and that $g(x) = \elemsk{h(x)}{S}$ for any $S \in A_p$ such that $x
\in p(S)$.

We verify that~\eqref{eq:psi-square} commutes by checking that the
corresponding square of inverse image maps does. On one hand, starting
with a copen section $\baseN{R}$ in $\St{A_p}$, we have
\begin{equation*}
  \textstyle
  p^{-1}(h^{-1}(\baseN{R})) = 
  p^{-1}(R) = \bigcup \,\set{S \such p(S) \subseteq R},
\end{equation*}
where $S$ in the union ranges over copen sections in~$E$. On the other
hand,
\begin{multline*}
  \textstyle
  g^{-1}(q_{A_p}^{-1}(\baseN{R})) =
  g^{-1}(\bigcup \set{\baseM{S} \such p(S) \subseteq R}) = \\
  \textstyle
  \bigcup \, \set{g^{-1}(\baseM{S}) \such p(S) \subseteq R} =
  \bigcup \, \set{S \such p(S) \subseteq R},
\end{multline*}
where $S$ again ranges over copen sections in~$E$.

Naturality of $\psi$ involves the commutativity of a cube which we
prefer not to draw because six of its faces commute by definition and
the two remaining faces are
\begin{equation*}
  \xymatrix@+1.5em{
    {B} \ar[r]^(0.4){\down{(\psi_p)}} \ar@{-^>}[d]_h
    &
    {\St{A_p}} \ar@{-^>}[d]^{\down{\homo{(g,h)}}}
    \\
    {B'} \ar[r]_(0.4){\down{(\psi_{p'})}}
    &
    {\St{A_{p'}}}
  }
  \qquad\qquad
  \xymatrix@+1.5em{
    {E} \ar[r]^(0.4){\up{(\psi_p)}} \ar@{-^>}[d]_g
    &
    {\Sk{A_p}} \ar@{-^>}[d]^{\up{(\homo{(g,h)})}}
    \\
    {E'} \ar[r]_(0.4){\up{(\psi_{p'})}}
    &
    {\Sk{A_{p'}}}
  }
\end{equation*}
where $p : E \to B$ and $p' : E' \to B'$ are skew Boolean spaces and
$(g,h)$ is a morphism from $p$ to $p'$. We check commutativity of the
right-hand square, the other one is similar. Again, we verify that the
corresponding square of inverse image maps commutes. For any copen
section $\baseM{S'}$ in the lower-right corner we have
\begin{equation*}
  g^{-1}(\up{(\psi_{p'})}^{-1}(\baseM{S'})) =
  g^{-1}(S')
\end{equation*}
and
\begin{equation*}
  \up{(\psi_p)}^{-1}(\up{(\homo{(g,h)})}^{-1}(\baseM{S'})) =
  \up{(\psi_p)}^{-1}(\baseM{\homo{(g,h)}(S')}) =
  \homo{(g,h)}(S') =
  g^{-1}(S').
\end{equation*}
We have proved the following main theorem.

\begin{theorem}
  \label{th:right-handed-duality}
  The category of right-handed skew Boolean algebras with intersections is dual
  to the category of skew Boolean spaces.
\end{theorem}

\noindent
Clearly, there is also duality between \emph{left}-handed skew algebras and skew
Boolean spaces, simply because the categories of left-handed and the
right-handed skew algebras are isomorphic.

\subsection{Duality for skew algebras}
\label{sec:dual-skew-algebr}

To see what is needed for duality in the case of a general skew algebra,
consider what happens when we take a skew algebra $A$ to its skew Boolean space $q_A : \Sk{A} \to \St{A}$, and then the space to the right-handed skew
algebra $A_{q_A}$. By Proposition~\ref{prop:baseM-properties} the elements and
the natural partial order do not change (up to isomorphism), but the operations
do. The new ones are expressed in terms of the original ones as
\begin{equation*}
  x \land' y = y \land x \land y 
  \qquad\text{and}\qquad
  x \lor' y = x \lor y \lor x .
\end{equation*}
If we want to recover $\land$ and $\lor$ from $\land'$ and $\lor'$ we need to
break the symmetry that is present in $\land'$ and $\lor'$ by keeping around
enough information about the original operations of~$A$.

Recall that a rectangular band is a set with an operation which is idempotent,
associative and it satisfies the rectangular identity. We can similarly define
rectangular bands in any category with finite products. For example, a
rectangular band in the category of étale maps over a given base space~$B$ is an
étale map $p : E \to B$ together with a continuous map ${\land} : E \times_B E
\to E$ over~$B$ which satisfies the required identities fiber-wise.

The skew Boolean space $q_A : \Sk{A} \to \St{A}$ carries the structure of a
rectangular band whose operation ${\sand} : \Sk{A} \times_{\St{A}} \Sk{A} \to
\Sk{A}$ is defined by
\begin{equation*}
  \elemsk{P}{a} \sand \elemsk{P}{b} = \elemsk{P}{a \land b}.
\end{equation*}
Idempotency and associativity of $\sand$ follow immediately from the
corresponding properties of $\land$ and the fact that $\thetaeq_P$ is a
congruence. To see that the rectangular identity is satisfied, let $a, b, c
\not\in P$. Because $(A, {\land})$ forms a normal band, namely it satisfies the
identity $x \land y \land z \land w = x \land z \land y \land w$, it follows
that $a \land b \land c \leq a \land c$ and so $a \land b \land c \thetaeq_P a
\land c$ by Lemma~\ref{lemma:s-less-than-t-not-in-P}, which is equivalent to
$\elemsk{P}{a} \sand \elemsk{P}{b} \sand \elemsk{P}{c} = \elemsk{P}{a} \sand
\elemsk{P}{c}$.

We have to check that $\sand$ is continuous. Let $(\elemsk{P}{a},
\elemsk{P}{b})$ be any point of the domain of~$\sand$ and suppose $\elemsk{P}{a
  \land b} \in \baseM{c}$ for some $c \in A$. We seek an open neighborhood of
$(\elemsk{P}{a}, \elemsk{P}{b})$ which is mapped into $\baseM{c}$ by $\sand$.
Because basic open subsets of $\Sk{A} \times_{\St{A}} \Sk{A}$ are of the form
\begin{equation*}
  \set{(\elemsk{Q}{u}, \elemsk{Q}{v}) \such
    \text{$Q \subseteq A$ prime ideal, $u \not\in Q$ and $v \not\in Q$}},
\end{equation*}
it suffices to find $u, v \in A$ such that $a \thetaeq_P u$ and $b \thetaeq_P
v$, and for all prime ideals $Q \subseteq A$, if $u \not\in Q$ and $v \not\in Q$
then $u \land v \thetaeq_Q c$. We claim that
\begin{align*}
  u &= (a \land b \land a) \cap (c \land a) \\
  v &= (b \land a \land b) \cap (b \land c)
\end{align*}
satisfy these conditions. We note that $u \land v \leq c$ because
\begin{equation*}
  u \land v =
  ((a \land b \land a) \cap (c \land a) ) \land
  ((b \land a \land b) \cap (b \land c)) \leq (c\land a) \land (b\land c)
  \leq c,
\end{equation*}
where we used the fact that $\land$ is compatible with the natural partial
order, which is the case because $(A, {\land})$ is a normal band. Next observe
that $a \land b \land a \thetaeq_P c \land a$ because $\elemsk{P}{a \land b} \in \baseM{c}$,
hence
\begin{equation*}
  u = (a \land b \land a) \cap (c \land a)
    \thetaeq_P (a \land b \land a) \cap (a \land b \land a) = a \land b \land a
    \thetaeq_P a,
\end{equation*}
where the last step follows from Lemma~\ref{lemma:s-less-than-t-not-in-P} and
$a, b \not\in P$. We similarly show that $v \thetaeq_P b$. If $Q \subseteq A$ is
a prime ideal such that $u \not\in Q$ and $v \not\in Q$, then $u \land v \not\in
Q$ and since $u \land v \leq c$ we get $u \land v \thetaeq_Q c$, again by
Lemma~\ref{lemma:s-less-than-t-not-in-P}.

The rectangular band structure on $A_p$ is precisely what is needed for duality
in the general case.

\begin{theorem}
  The category of skew Boolean algebras with intersections is dual to the
  category of rectangular skew Boolean spaces.
\end{theorem}

\begin{proof}
  By rectangular skew Boolean space $(p : E \to B, {\sand})$ we mean a
  rectangular band in the category of surjective étale maps over~$B$. More
  precisely, it is a skew Boolean space $p : E \to B$ together with a (not
  necessarily proper) continuous map $\sand$ over~$B$
  \begin{equation*}
    \xymatrix{
      {E \times_B E} \ar[dr] \ar[rr]^{\sand} & &
      {E} \ar[ld]
      \\
      & {B} &
    }
  \end{equation*}
  which makes every fiber $E_x$ into a rectangular band. Notice that in general
  a rectangular skew Boolean space is \emph{not} a rectangular band in the
  category of skew Boolean spaces because $\sand$ need not be proper. A morphism
  between $(p : E \to B, {\sand})$ and $(p' : E' \to B', {\sand'})$ is a
  morphism of skew Boolean spaces $(g,h) : p \to p'$ which commutes with the
  operations on its domain of definition:
  \begin{equation*}
    \xymatrix@C+1.5em{
      {\dom{g} \times_B \dom{g}}
      \ar@{-^>}[r]^(0.55){g \times g}
      \ar[d]_{\sand}
      &
      {E' \times_{B'} E'}
      \ar[d]^{\sand'}
      \\
      {\dom{g}}
      \ar@{-^>}[r]_{g}
      &
      {E'}
    }
  \end{equation*}
  Note that the commutativity of the square implies that $\dom{g}$ is closed
  under~$\sand$, so $\dom{g} \cap E_x$ is a rectangular sub-band of $E_x$ at
  every $x \in B$. And since~$g$ is bijective on fibers, $\restricted{g}{x} :
  \dom{g} \cap E_x \to E_{h(x)}$ is an isomorphism of rectangular bands for
  every $x \in \dom{h}$. We denote the category of rectangular skew Boolean
  spaces and their morphisms by~$\CatRS$.

  The duality is witnessed by a pair of contravariant functors
  \begin{equation*}
    \FunS : \op{\CatA} \to \CatRS
    \qquad\text{and}\qquad
    \FunA : \op{\CatRS} \to \CatA.
  \end{equation*}
  The functor $\FunS$ maps a skew algebra $A$ to the rectangular skew Boolean
  space $\FunS(A) = (p : \Sk{A} \to \St{A}, {\sand})$, as described above. It takes
  a morphism $f : A \to A'$ to the morphism of skew Boolean spaces $\FunS(f) =
  (\up{f}, \down{f})$, which commutes with $\sand$ because~$f$ commutes with
  $\land$.

  The functor $\FunA$ maps a rectangular skew Boolean space $(p : E \to B,
  {\sand})$ to the skew algebra $\FunA(p, {\sand})$ whose elements are the copen
  sections of~$p$ and the operations are defined as follows:
  \begin{align*}
    0 &= \emptyset,\\
    S \land R &=  (S \cap \sigma_p(R)) \sand (\sigma_p(S) \cap R), \\
    S \lor R  &=  (S - \sigma_p(R)) \cup (R - \sigma_p(S)) \cup (R \land S), \\
    S \setminus R &= S - \sigma_p(R), \\
    S \cap R &= S \cap R,
  \end{align*}
  These form a skew algebra because they are restrictions of the operations from
  Theorem~\ref{theorem:partial-functions-band}.

  The functor $\FunA$ maps a morphism $(g,h) : (p : E \to B, {\sand}) \to (p':E'
  \to B', {\sand'})$ to the homomorphism $\FunA(g,h) = \homo{(g,h)}$. We need to
  verify that $\FunA(g,h)$ preserves~$\sand$. Recall that $\homo{(g,h)}$ is just
  $g^{-1}$ acting on copen sections. In Section~\ref{sec:spaces-to-algebras} we
  checked that $g^{-1}$ commutes with the saturation operations. Because for
  every $x \in \dom{h}$ the map $\restricted{g}{x} : E_x \cap \dom{g} \to
  E'_{h(x)}$ is an isomorphism of rectangular bands, it is not hard to see that
  $g^{-1}$ commutes with~$\sand$. Therefore, $g^{-1}$ commutes with all the
  operations used to define the operations on $G(p, {\sand})$ and $G(p',
  {\sand'})$, so it is a homomorphism of skew algebras.

  It remains to be checked that $\FunS \circ \FunA$ and $\FunA \circ \FunS$ are
  naturally isomorphic to identity functors. Luckily, we can reuse a great deal
  of verification of duality for right-handed algebras from
  Section~\ref{sec:skew-rh-duality-proofs}. 

  The natural isomorphism $\phi$ from the identity to $\FunA \circ \FunS$ is
  defined as in the right-handed case: for a skew algebra $A$ set $\phi_A(a) =
  \baseM{a}$. Thus we already know that it is a bijection which preserves
  intersections and relative complements, but we still have to check that it
  preserves meets and joins. It preserves meets because
  \begin{equation*}
    \baseM{a} \land \baseM{b} =
    \baseM{a \land b \land a} \sand \baseM{b \land a \land b} =
    \baseM{a \land b}
  \end{equation*}
  where we used Proposition~\ref{prop:baseM-properties} in the first step and
  the fact that $\elemsk{P}{a \land b \land a} \sand \elemsk{P}{b \land a \land
    b} = \elemsk{P}{a \land b}$ in the last step.
  With the help of Proposition \ref{prop:baseM-properties} it is not hard to
  verify that whenever $a$ and $b$ commute then $\baseM{a \lor b} = \baseM{a}
  \cup \baseM{b} = \baseM{b \lor a}$, so $\phi_A$ preserves commuting joins. But
  since for arbitrary $a$ and $b$ their join can be expressed as a commuting
  join $a \lor b = (a \setminus b) \lor (b\setminus a) \lor (b \land a)$, and we
  already know that $\phi_A$ preserves $\setminus$ and $\land$, it follows that
  $\phi_A$ preserves joins. Naturality of $\phi_A$ is checked as in the
  right-handed case.

  The natural isomorphism $\psi$ from the identity to $\FunS \circ \FunA$ is
  defined as in the right-handed case. Given a rectangular skew Boolean space
  $(p : E \to B, {\sand})$, let $\psi_{p, {\sand}}$ be the morphism consisting
  of the two homeomorphisms $\down{(\psi_{p,{\sand}})} = h$ and
  $\up{(\psi_{p,{\sand}})} = g$ from diagram~\eqref{eq:psi-square}. All that we
  need to check in addition to what was already checked for~$\psi$ in
  Section~\ref{sec:skew-rh-duality-proofs} is that $g$ preserves the rectangular
  band structure. For any $b \in B$, $x, y \in E_b$ and $T \in A_p$ we have
  \begin{equation*}
    g(x \sand y) \in \baseM{T} \iff
    x \sand y \in T.
  \end{equation*}
  On the other hand, if $g(x) = \elemsk{h(b)}{S}$ and $g(y) = \elemsk{h(b)}{R}$
  then 
  \begin{multline*}
    g(x) \sand g(y) \in \baseM{T} \iff
    \elemsk{h(b)}{S \sand R} \in \baseM{T} \iff \\
    S \sand R \thetaeq_{h(b)} T \ \text{and}\ b \in p(T) \iff
    x \sand y \in T.
  \end{multline*}
  We see that $g(x) \sand g(y)$ and $g(x \sand y)$ have the same
  neighborhoods, therefore they are equal.
\end{proof}

It may be argued that our duality has not gone all the way from algebra to
geometry because a rectangular skew Boolean space still carries the algebraic
structure of a rectangular band. However, this is not really an honest algebraic
structure, as can be suspected from the fact that the category of non-empty
rectangular bands is equivalent to the category of pairs of sets. The
equivalence takes a rectangular band $(A, {\land})$ to the pair of sets $(A/\RR,
A/\LL)$ where $A/\RR$ and $A/\LL$ are the quotients of~$A$ by Green's relations
$\RR$ and $\LL$, respectively. In the other direction, a pair of sets $(X,Y)$ is
mapped to the rectangular band $X \times Y$ with the operation $(x_1, y_1) \land
(x_2, y_2) = (x_1, y_2)$. The analogous decomposition of rectangular skew
Boolean spaces yields the following variant of duality for skew algebras.

\begin{theorem}
  The category of skew Boolean algebras with intersections is dual to the
  category of pairs of skew Boolean spaces with common base.
\end{theorem}

\begin{proof}
  A pair of skew Boolean spaces with a common base is a diagram
  \begin{equation}
    \label{eq:skew-pair-object}
    \xymatrix{
      {E_L} \ar[r]^{p_L}
      &
      {B}
      &
      {E_R} \ar[l]_{p_R}
    }
  \end{equation}
  where $p_L : E_L \to B$ and $p_R : E_R \to B$ are skew Boolean spaces. A
  morphism is a commutative diagram
  \begin{equation}
    \label{eq:skew-pair-morphism}
    \xymatrix{
      {E_L} \ar[r]^{p_L} \ar@{-^>}[d]_{g_L} &
      {B} \ar@{-^>}[d]_{h} &
      {E_R} \ar[l]_{p_R} \ar@{-^>}[d]^{g_R} \\
      {E'_L} \ar[r]_{p'_L} &
      {B'} &
      {E'_R} \ar[l]^{p'_R}
    }
  \end{equation}
  in which the left- and right-hand square are morphisms of skew Boolean spaces
  (in the vertical direction). The diagrams are composed in the obvious way and
  we clearly get a category.

  We establish the duality by showing that the category of pairs of skew Boolean
  spaces with common base is equivalent to the category of rectangular skew
  Boolean spaces. The idea is to have equivalence functors work at the level of
  fibers in the same way as the equivalence of non-empty rectangular bands and
  pairs of sets.

  To convert a pair of skew Boolean spaces~\eqref{eq:skew-pair-object} into a
  rectangular Boolean space we form the pullback
  \begin{equation*}
    \xymatrix{
      E \ar[d]  \ar@{->}[r] \pbcorner
      &
      E_R \ar[d]^{p_R}
      \\
      E_L \ar@{->}[r]_{p_L}
      &
      B
    }
  \end{equation*}
  to obtain a skew Boolean space $p : E \to B$. Concretely, the fiber $E_x$ over
  $x \in B$ consists of pairs $(u,v) \in E_L \times E_R$ such that $p_L(u) = x =
  p_R(v)$, and $p(u,v) = p_L(u) = p_R(v)$. The rectangular band operation
  $\sand$ on $p : E \to B$ defined by
  \begin{equation*}
    (u_1, v_1) \sand (u_2, v_2) = (u_1, v_2),
  \end{equation*}
  obviously makes the fibers into rectangular bands. A
  morphism~\eqref{eq:skew-pair-morphism} corresponds to the morphism
  \begin{equation*}
    \xymatrix{
      {E}
      \ar@^{->}[r]^{g}
      \ar[d]_{p}
      &
      {E'}
      \ar[d]^{p'}
      \\
      {B}
      \ar@^{->}[r]_{h}
      &
      {B'}
    }
  \end{equation*}
  where $g$ is the partial map with domain $\dom{g_L} \times_B \dom{g_R}$
  defined by
  \begin{equation*}
    g(x,y) = (g_L(x), g_R(y)).
  \end{equation*}
  It clearly preserves $\sand$.

  In the opposite direction we start with a rectangular skew Boolean space $(p :
  E \to B, {\sand})$ and form a pair of skew Boolean spaces with a common base
  as follows. First construct the fiber-wise Green's relation $\LL$ on $E$ as
  the equalizer
  \begin{equation*}
    \xymatrix@C+4em{
      {\LL} \ar[r]^(0.4){\ell} \ar[rd]
      &
      {E \times_B E}
      \ar[d]
      \ar@<+0.25em>[r]^{\id{E \times_B E}}
      \ar@<-0.25em>[r]_{{\sor} \times_B {\sand}}
      &
      {E \times_B E}
      \ar[ld]
      \\
      & {B} &
    }
  \end{equation*}
  in the topos of étale maps with base~$B$. In the above diagram $\sor : E
  \times_B E \to E$ is the operation associated with $\sand$ by $x \sor y = y
  \sand x$. Still in the topos, we form the coequalizer
  \begin{equation*}
    \xymatrix@C+3em{
      {\LL}
      \ar@<+0.25em>^{\pi_1 \circ \ell}[r]
      \ar@<-0.25em>_{\pi_2 \circ \ell}[r]
      \ar[rd]
      &
      {E} \ar[r]^{q_L} \ar[d]_{p}
      &
      {E_L} \ar[ld]^{p_L}
      \\
      & {B} &
    }
  \end{equation*}
  The quotient $E_L$ is Hausdorff because by
  Proposition~\ref{prop:etale-equalizers-coequalizers} the map $q_L$ is open,
  and a pair of points in~$E$ may always be separated by clopen sections. We now
  have one of the skew Boolean spaces $p_L : E_L \to B$, and there is an
  analogous construction of $p_R : E_R \to B$.
  On a single fiber $E_x$ over $x \in B$ the functor just performs the usual
  decomposition of the rectangular band $E_x$ into its left- and right-handed
  factors $E_x/\RR_x$ and $E_x/\LL_x$. This is so because by
  Proposition~\ref{prop:etale-equalizers-coequalizers} equalizers and
  coequalizers of étale maps are computed fiber-wise.

  A morphism $(g,h)$ from $(p : E \to B, {\sand})$ to $(p' : E' \to B',
  {\sand'})$, displayed explicitly as inclusions of the domains of definition
  and total maps,
  \begin{equation}
    \xymatrix{
      {E} \ar[d]_{p}   &
      {\dom{g}} \ar@{->}[l] \ar[d] \ar[r]^(0.55){g} &
      {E'} \ar[d]^{p'} \\
      {B} & {\dom{h}} \ar@{->}[l] \ar[r]_(0.55){h} & {B'}
    }
  \end{equation}
  corresponds to a morphism between pairs of skew Boolean spaces as described
  next. Consider the commutative diagram
  \begin{equation*}
    \xymatrix@C+1.5em{
      {\LL \cap (\dom{g} \times_B \dom{g})}
      \ar@<+0.25em>[r]
      \ar@<-0.25em>[r]
      \ar[rd]
      &
      {\dom{g}} \ar[r]^(0.45){q_L} \ar[d]_{p}
      &
      {q_L(\dom{g})} \ar@{ (->}[d]
      \\
      & {B}
      & {E_L} \ar[l]^{p_L}
    }
  \end{equation*}
  where the two parallel arrows into $\dom{g}$ are the restrictions of $\pi_1
  \circ l$ and $\pi_2 \circ l$ to $\dom{g}$. By
  Proposition~\ref{prop:etale-equalizers-coequalizers} the map $q_L$ is open,
  hence $q_L(\dom{g})$ is an open subspace of $E_L$. Moreover, because $\dom{g}$
  is an open subspace of~$E$ and $q_L$ is open, it is not hard to check that the
  top row of the diagram is a coequalizer.
  Because~$g$ commutes with $\sand$, the map $q'_L \circ g$ factors through the
  coequalizer,
  \begin{equation*}
    \xymatrix@C+2em{
      {\dom{g}}
      \ar[r]^{g}
      \ar[d]_{q_L}
      &
      {E'}
      \ar[d]^{q'_L}
      \\
      {q_L(\dom{g})}
      \ar[r]_(0.55){g_L}
      &
      {E'_L}
    }
  \end{equation*}
  We have obtained a partial map $g_L : E_L \parto E'_L$ whose domain
  $q_L(\dom{g})$ is an open subspace of $E_L$. Also $q_L$ is proper because $g$
  is proper. We similarly obtain the right-handed version $g_R : E_R \parto
  E'_R$. This gives us the desired morphism
  \begin{equation*}
    \xymatrix{
      {E_L} \ar[r]^{p_L} \ar@{-^>}[d]_{g_L} &
      {B} \ar@{-^>}[d]^{h}&
      {E_R} \ar[l]_{p_R} \ar@{-^>}[d]^{g_R}  \\
      {E'_L} \ar[r]_{p'_L} &
      {B'} &
      {E'_R} \ar[l]^{p'_R}
    }
  \end{equation*}
  To see that the two functors just described form an equivalence, we use the
  fact that fiber-wise they correspond to the equivalence between non-empty
  rectangular bands and pairs of non-empty sets. We omit the details.
\end{proof}

\section{Variations}
\label{sec:variations}

The morphisms between skew Boolean spaces were determined by our taking
\emph{all} homomorphisms on the algebraic side of duality. In this section we
consider several variants in which the homomorphisms are restricted. We limit
attention to the right-handed case, and ask the kind reader who will work out
the general case to let us know whether there are any surprises. 
As a preparation we first show how morphisms of skew Boolean spaces decompose
into open inclusions and pullbacks.

\begin{lemma}
  \label{lemma:bijective-on-fibers-homeo}
Suppose $p : E \to B$ and $p' : E' \to B$ are skew Boolean
spaces and $g : E
  \to E'$ is a proper continuous map such that
  \begin{equation*}
    \xymatrix{
      {E} \ar[rr]^{g} \ar[rd]_{p}
      & &
      {E'} \ar[ld]^{p'}
      \\
      &
      {B}
      &
    }
  \end{equation*}
commutes. If $g$ is bijective on fibers then it is a
homeomorphism.
\end{lemma}

\begin{proof}
  It is obvious that $g$ is a bijection, so we only need to check that it is a
  closed map. If $K \subseteq E'$ is compact then the restriction
  $\restricted{g}{g^{-1}(K)} : g^{-1}(K) \to K$ is a closed map because it maps
  from the compact space $g^{-1}(K)$ to the Hausdorff space $E'$. Therefore, if
  $F \subseteq E$ is closed then $g(F) \cap K = \restricted{g}{g^{-1}(K)}(F \cap
  g^{-1}(K))$ is closed in $K$ for every compact $K \subseteq E'$. Because $E'$
  is locally compact, it is compactly generated and we may conclude that $g(F)$
  is closed.
\end{proof}

\begin{lemma}
  \label{lemma:bijective-on-fibers-pullback}
Suppose $p : E \to B$ and $p ': E' \to B'$ are skew Boolean
spaces and $g :
  E \to E'$ is a proper continuous map. A commutative square
  \begin{equation*}
    \xymatrix{
      {E} \ar[r]^{g} \ar[d]_{p}
      &
      {E'} \ar[d]^{p'}
      \\
      {B} \ar[r]_{h}
      &
      {B'}
    }
  \end{equation*}
  is a pullback if, and only if, $g$ is bijective on fibers.
\end{lemma}

\begin{proof}
  It is easy to check that $g$ is bijective on fibers if the square is a
  pullback. Conversely, suppose $g$ is bijective on fibers. We form the pullback
  of $h$ and $p'$ and obtain a factorization $e$, as in the diagram
  \begin{equation*}
    \xymatrix{
      {E} \ar[r]^{e} \ar[rd]_{p} \ar@/^2em/[rr]^{g} &
      {P} \ar[r]^{q} \ar[d] \pbcorner &
      {E'} \ar[d]^{p'} \\
      & {B} \ar[r]_{h} & {B'}
    }
  \end{equation*}
  The map $e$ is proper because $g$ is proper. Indeed, if $K \subseteq P$ is
  compact then $e^{-1}(K)$ is a closed subset of the compact subset
  $g^{-1}(q(K))$, therefore it is compact. Furthermore, $e$ is bijective on
  fibers because $g$ and $q$ are. By Lemma~\ref{lemma:bijective-on-fibers-homeo}
  the map $e$ is a homeomorphism, therefore the outer square is a pullback.
\end{proof}

Consider a morphism of Boolean spaces, with inclusion of domains displayed explicitly:
\begin{equation}
  \label{eq:morphism-decomposition}
  \xymatrix@C+1em{
    {E} \ar[d]_{p}   &
    {\dom{g}} \ar[r] \ar@{->}[l] \ar[d] \ar[r]^(0.6){g} &
    {E'} \ar[d]^{p'} \\
    {B} & {\dom{h}} \ar@{->}[l] \ar[r]_(0.6){h} & {B'}
  }
\end{equation}
The left square need not be a morphism in our category because inclusions of
open subsets need not be proper. If we turn them around they become
\emph{partial identities} with open domains of definitions, and we do get a
decomposition
\begin{equation}
  \label{eq:morphism-decomposition-other-direction}
  \xymatrix@C+1em{
   {E} \ar[d]_{p}  \ar@^{->}[r] &
    {\dom{g}}  \ar[d] \ar[r]^(0.6){g} &
    {E'} \ar[d]^{p'} \\
    {B} \ar@^{->}[r] & {\dom{h}}  \ar[r]_(0.6){h} & {B'}
  }
\end{equation}
in which both squares are morphisms between skew Boolean spaces. Now by
Lemma~\ref{lemma:bijective-on-fibers-pullback} the condition that~$g$ is
bijective on fibers is equivalent to the right square being a pullback.
Therefore, every morphism can be decomposed into a partial identity (with open
domain of definition) and a pullback. What does the
decomposition~\eqref{eq:morphism-decomposition-other-direction} correspond to on
the algebraic side of duality? In order to answer the question, we need to study
a certain kind of ideals in skew algebras.

A \emph{$\leq$-ideal} of a skew algebra $A$ is a subset $I \subseteq A$ which is
closed under finite joins and the natural partial order~$\leq$. In particular,
$I$ is nonempty as it contains the empty join~$0$. A $\leq$-ideal may
equivalently be described as a subalgebra that is closed under the natural
partial order because in a skew algebra we always have $x \land y \leq y \lor
x$. The following Lemma gives an explicit description of the $\leq$-ideal
generated by a given subset, akin to how sets generate ideals in rings.

\begin{lemma}\label{lemma:leq-ideals}
  The $\leq$-ideal $\leqideal{S}$ generated by a subset $S \subseteq A$ is
  formed as the closure by finite joins of the downard closure of~$S$ with
  respect to the natural partial order:
  \begin{equation*}
    \leqideal{S} = \set{
      x_1 \lor \cdots \lor x_n \such
      \forall i \leq n \,.\,
      \exists y_i \in S \,.\, x_i \leq y_i
      }.
  \end{equation*}
\end{lemma}

\begin{proof}
  Because any ideal that contains $S$ also contains $\leqideal{S}$ we only have
  to check that $\leqideal{S}$ is a $\leq$-ideal. The set $\leqideal{S}$ is
  obviously closed under joins. To see that it is closed under the natural
  partial order, let $x \leq x_1 \lor \cdots \lor x_n$ where $x_i \leq y_i$ and
  $y_i \in S$. Then
  \begin{multline*}
    x = (x_1 \lor \cdots \lor x_n) \land x \land (x_1 \lor \cdots \lor x_n) = \\
    (x_1 \land x \land x_1) \lor \cdots \lor (x_n \land x \land x_n),
  \end{multline*}
  where we canceled all terms of the form $x_i \land x \land x_j$ with $i\neq j$, by
  the usual argument in the skew lattice theory that amounts to the fact that
  $(A, {\lor})$ is \emph{regular} as a band, i.e., it satisfies the identity $a \lor
  b \lor a \lor c \lor a = a \lor b \lor c \lor a$. Now, $x_i \land x \land x_i \leq
  x_i \leq y_i $ for all $i \leq n$ and thus $x \in \leqideal{S}$.
\end{proof}

\noindent
Note that in the previous lemma the order of operations matters. If we first
close under joins and then perform the downward closure we need not get a
$\leq$-ideal because the resulting set need not be closed under joins.

We say that a subset $S \subseteq A$ of a skew algebra $A$ is
\emph{$\leq$-cofinal} when $\leqideal{S} = A$. A homomorphism is
$\leq$-cofinal when its image is $\leq$-cofinal. If $A$ is commutative,
$S \subseteq A$ is $\leq$-cofinal precisely when it is cofinal in the usual
sense: for every $x \in A$ there is $y \in S$ such that $x \leq y$. 

Every homomorphism $f : A \to A'$ of skew algebras may be decomposed into a
$\leq$-cofinal morphism and an inclusion of a $\leq$-ideal
\begin{equation*}
  \xymatrix{
    {A} \ar[rr]^{f} \ar[rd]_{f}
    & &
    {A'}
    \\
    &
    {\leqideal{\im{f}}}
    \ar[ur]_{i}
    &
  }
\end{equation*}
The following two propositions show that the decomposition is dual to the
decomposition~\eqref{eq:morphism-decomposition-other-direction} on the
topological side of duality.

\begin{proposition}
\label{prop:partial_identities-vs-leq-ideals}
  Partial identities with open domains of definition on the topological side are
  dual to inclusions of $\leq$-ideals on the algebraic side.
\end{proposition}

\begin{proof}
  Let $p : E \to B$ be a skew Boolean space with open subsets $U \subseteq E$
  and $V \subseteq B$ such that $p(U) = V$. These determine a morphism of skew
  Boolean spaces
  \begin{equation*}
    \xymatrix@+1em{
      {E}
      \ar@^{->}[r]^{i}
      \ar[d]_{p}
      &
      {U}
      \ar[d]^{\restricted{p}{U}}
      \\
      {B}
      \ar@^{->}[r]_{j}
      &
      {V}
    }    
  \end{equation*}
  where $i$ and $j$ are the identity maps restricted to~$U$ and~$V$,
  respectively. It is easy to check that the corresponding homomorphism $f =
  \homo{(i,j)} : A_{\restricted{p}{U}} \to A_p$ is the inclusion of the
  subalgebra $A_{\restricted{p}{U}}$ into $A_p$. Its image is downward closed
  with respect to $\leq$ because the partial order in $A_p$ is inclusion of
  copen sections.

  Conversely, let $I \subseteq A$ be a $\leq$-ideal in $A$ and $i: I \to A$ the
  inclusion. The dual of $i$ is the morphism of Boolean spaces
  \begin{equation*}
    \xymatrix{
      {\Sk{A}} \ar@{-^>}[r]^{\up{i}} \ar[d]_{q_A}
      &
      {\Sk{I}} \ar[d]^{q_I}
      \\
      {\St{A}} \ar@{-^>}[r]_{\down{i}}
      &
      {\St{I}}
    }
  \end{equation*}
  where $\up{i}$ acts as $\up{i}(\elemsk{P}{a}) = \elemsk{P \cap I}{a}$ and is
  defined for any $a \in I$ and a prime ideal $P \subseteq A$ such that $a
  \not\in P$. The domain of $\up{i}$ is open because it is the union of those
  basic copen sets $\baseM{a}$ for which $a \in I$. The map $\up{i}$ is open
  because it takes $\baseM{a} \subseteq \Sk{A}$ to $\baseM{a} \subseteq \Sk{I}$.
  Therefore, $\up{i}$ really is (isomorphic to) a partial identity with an open
  domain of definition. The same fact for $\down{i}$ follows easily.
 \end{proof}

\begin{proposition}
  \label{prop:pullbacks-vs-leq-cofinal}
  Pullbacks on the topological side are dual to $\leq$-cofinal homomorphisms on
  the algebraic side.
\end{proposition}

\begin{proof}
  Consider a morphism of skew Boolean spaces that is a pullback, which is
  equivalent to it being defined everywhere:
  \begin{equation*}
    \xymatrix{
      {E} \pbcorner \ar[r]^{g} \ar[d]_{p}
      &
      {E'} \ar[d]^{p'}
      \\
      {B} \ar[r]_{h}
      &
      {B'}      
    }
  \end{equation*}
  To see that the corresponding homomorphism $f = \homo{(g,h)}:A_{p'} \to A_p$
  is $\leq$-cofinal, let $S$ be a copen section in $E$. Because $g$ is
  everywhere defined and $S$ is compact, there exist finitely many copen
  sections $R_1, \ldots, R_n$ in $E'$ such that $S$ is covered by the copen
  sections $g^{-1}(R_1), \ldots, g^{-1}(R_n)$. Let $T_1, \ldots, T_n$ be defined
  by $T_1=g^{-1}(R_1)\cap S$ and $T_{i+1}=(g^{-1}(R_{i+1})\cap S)\setminus T_i$ for $i\geq 1$.
  Then $S=T_1\cup \ldots \cup T_n = T_1\lor \ldots \lor T_n$ where the latter
  equlity follows because distinct $T_i$'s have disjoint saturations.
  Furthermore, each $T_i$ is contained in $g^{-1}(R_i)=f(R_i)$ and thus $S$ lies
  in $\leqideal{\im{f}}$.

  Next, let $f:A\to A'$ be a $\leq$-cofinal homomorphism of right-handed skew
  algebras. We claim that $\up{f}:\Sk{A'} \to \Sk{A}$ is everywhere defined. To
  see this take any $\elemsk{P}{b}\in \Sk{A'}$. Since $f$ is $\leq$-cofinal there
  exist $b_1,\ldots ,b_n\in A'$ and $a_1,\ldots , a_n \in A$ such that $b=b_1\lor
  \ldots \lor b_n$ and $b_i\leq f(a_i)$ for all $i=1,\dots, n$. Since $b\not \in
  P$ there exists $i$ such that $b_i\not \in P$. Thus $b_i\thetaeq _P f(a_i) $
  by Lemma \ref{lemma:s-less-than-t-not-in-P}. Let $\set{i_1, \ldots i_k} $ be
  the set of those indices $i_j$ that satisfy $b_{i_j}\not\in P$. Then
  $b\thetaeq_P f(a_{i_1})\lor \ldots \lor f(a_{i_k})=f(a_{i_1}\lor\dots \lor
  a_{i_k}) $ and thus $\elemsk{P}{b}\in \dom{\up{f}}$.
\end{proof}

Let us call a morphism of skew Boolean spaces
\begin{equation*}
  \xymatrix{
    {E}
    \ar@^{->}[r]^{g}
    \ar[d]_{p}
    &
    {E'}
    \ar[d]^{p'}
    \\
    {B}
    \ar@^{->}[r]_{h}
    &
    {B'}
  }    
\end{equation*}
\emph{total} when both $g$ and $h$ are total, and \emph{semitotal} when~$h$ is
total. A morphism is total precisely when it is a pullback square. As direct consequences of  Propositions~\ref{prop:partial_identities-vs-leq-ideals} and \ref{prop:pullbacks-vs-leq-cofinal} we obtain the dualities stated by the following pair of theorems.

\begin{theorem}
The category of skew Boolean spaces and partial identities with open domains is dual to the category of right-handed skew algebras and inclusions of $\leq$-ideals.
\end{theorem}

\begin{theorem}
  The category of skew Boolean spaces and total morphisms is dual to the
  category of right-handed skew algebras and $\leq$-cofinal homomorphisms.
\end{theorem}

\noindent
The duality for the semitotal morphisms is similar to duality for total
morphisms, except that we have to replace $\leq$-cofinality with
$\preceq$-cofinality: a subset $S \subseteq A$ of a skew algebra is
\emph{$\preceq$-cofinal} when the ideal generated by~$S$ equals~$A$. Such an
ideal may be computed either as the $\lor$-closure of $\preceq$-closure of~$S$,
or as the $\preceq$-closure of $\lor$-closure of~$S$.

A homomorphism $f : A \to A'$ is $\preceq$-cofinal when its image is
$\preceq$-cofinal. Because the image is closed under finite joins,
$\preceq$-cofinality of $f$ amounts to the following condition: for every $y \in
A'$ there is $x \in A$ such that $y \preceq f(x)$. Cofinal homomorphisms between
commutative algebras are also known as \emph{proper} homomorphisms, but we avoid
this terminology because we already use the term proper on the topological side.
In a slightly different setup this kind of duality is considered by Ganna
Kudryavtseva~\cite{kudryavtseva11:_stone_boolean}.

\begin{theorem}
  The category of skew Boolean spaces and semitotal morphisms is dual to the
  category of right-handed skew algebras and $\preceq$-cofinal homomorphisms.
\end{theorem}

\begin{proof}
  Assume that $h$ in diagram~\eqref{eq:morphism-decomposition-other-direction}
  is total. For every $S\in B^*$ there exist $R_1, \ldots, R_n\in (B')^*$ such
  that $S$ (as a copen set in $B$) is covered by the copen set $h^{-1}(R_1\cup
  \ldots \cup R_n)$. Hence $S\leq h^*(R_1\lor \ldots \lor R_n)$.

  To prove the converse, assume that $f:A\to A'$ is a homomorphism of
  right-handed skew algebras that is cofinal with respect to $\preceq$. We claim
  that $\down{f}:\St{A'} \to \St{A}$ is everywhere defined. To see this take any
  $P\in \St{A'}$. Since $f$ is $\preceq$-cofinal it follows that $\im{f}$ is not
  contained in any proper ideal of $A'$. Hence there exists $a\in A$ such that
  $f(a)\not\in P$, which implies $f^{-1}(P)\neq A$. Thus $f^{-1}(P)$ is a prime
  ideal in $A$ and so $\down{f}$ is defined at $P$, namely
  $\down{f}(P)=f^{-1}(P)$.
\end{proof}

To get still more variations of duality we consider notions of
saturation. We say that a homomorphism $f : A \to A'$ of skew algebras
is \emph{$\DD$-saturated} when its image is saturated with respect to
Green's relation~$\DD$.

\begin{lemma}
  \label{lemma:D-two-saturations}
  A homomorphism between skew algebras is $\DD$-saturated if, and only if, it
  maps each $\DD$-class surjectively onto a $\DD$-class.
\end{lemma}

\begin{proof}
  The ``if'' part is obvious. For the ``only if'' part, let $f : A \to A'$ be a
  $\DD$-saturated homomorphism. Suppose $a \in A$ and $b \in A'$ such
  that $f(a) \Drel b$. Because $f$ is $\DD$-saturated there is $a' \in
  A$ such that $f(a') = b$. Consider
  \begin{equation*}
    a'' = (a' \land a \land a') \lor a \lor  (a' \land a \land a').
  \end{equation*}
  It is obvious that $a'' \Drel a$. Because $f(a) \Drel b$
  we get $b \land f(a) \land b = b$ and $b \lor f(a) \lor b = b$,
  which implies 
  \begin{equation*}
    f(a'') = 
    (b \land f(a) \land b) \lor f(a) \lor (b \land f(a) \land b) =
    b \lor f(a) \lor b = b,
  \end{equation*}
  as desired.
\end{proof}

There is a notion of saturation of morphisms on the topological side, too. Say
that $(g,h) : p \to p'$ from $p : E \to B$ to $p' : E' \to B'$ is
\emph{saturated} when the domain of $g$ is saturated with respect to $p$, i.e.,
$\dom{g} = p^{-1}(\dom{h})$. This is equivalent to the left square
in~\eqref{eq:morphism-decomposition} being a pullback. In a different context
such morphisms were called partial pullbacks by Erik Palmgren and Steve Vickers~\cite{Palmgren}.

The two notions of saturation are \emph{not} dual to each other, so
they yield two more dualities.

\begin{theorem}
  The category of right-handed skew algebras and $\DD$-saturated homomorphisms
  is dual to the category of skew Boolean spaces and those morphisms
  \begin{equation*}
    \xymatrix{
      {E}
      \ar@^{->}[r]^{g}
      \ar[d]_{p}
      &
      {E'}
      \ar[d]^{p'}
      \\
      {B}
      \ar@^{->}[r]_{h}
      &
      {B'}
    }    
  \end{equation*}
  that satisfy the following lifting property: for every copen $U \subseteq B'$ and
  copen section $S$ of $p$ above $h^{-1}(U)$ there exists a copen section $R$ of
  $p'$ above $U$ such that $S = g^{-1}(R)$.
\end{theorem}

\begin{proof}
  Let $f : A \to A'$ be a $\DD$-saturated homomorphism. To show that $(\up{f},
  \down{f}) : A_{p'} \to A_p$ has the desired property, consider a copen
  $\baseN{a} \subseteq \St{A}$ and a copen section $\baseM{b} \in \Sk{A'}$ above
  $\baseN{f(a)}$. Because $\baseN{f(a)} = \baseN{b}$, we have $f(a) \Drel b$ and
  so by Lemma~\ref{lemma:D-two-saturations} there exists $a' \in A$ such that
  $a' \Drel a$ and $b = f(a')$. So $\baseN{a} = \baseN{a'}$ and
  $\up{f}^{-1}(\baseM{a'}) = \baseM{f(a')} = \baseM{b}$.

  Conversely, suppose we have a morphism between skew Boolean spaces, as in the
  statement of the theorem. We need to show that the corresponding homomorphism
  $f = \homo{(g,h)} : A_{p'} \to A_p$ is $\DD$-saturated. Let $S \subseteq E'$ be a
  copen section of $p'$ above $V \subseteq B'$, and let $R \subseteq E$ be a
  copen section of $p$ above $h^{-1}(V)$, i.e., $R \Drel g^{-1}(S) = f(S)$. By the
  property of our morphism there exists a copen section $S' \subseteq E'$ above
  $V$ such that $f(S') = g^{-1}(S') = R$, as required.
\end{proof}

\begin{theorem}
  The category of skew Boolean spaces and saturated morphisms is dual to the
  category of right-handed skew algebras and those homomorphisms $f : A \to A'$
  for which $\leqideal{\im{f}}$ is closed under the natural preorder.
\end{theorem}

\begin{proof}
  Consider a saturated morphism
  \begin{equation*}
    \xymatrix{
      {E}
      \ar@^{->}[r]^{g}
      \ar[d]_{p}
      &
      {E'}
      \ar[d]^{p'}
      \\
      {B}
      \ar@^{->}[r]_{h}
      &
      {B'}
    }    
  \end{equation*}
  and let $f = \homo{(g,h)} : A_{p'} \to A_p$ be the corresponding homomorphism.
  Let $R$ be a section of $p'$ above a copen $V \subseteq B'$ and $S$ a copen
  section of~$p$ above a copen~$U \subseteq h^{-1}(V)$. Because $\dom{g}$ is
  saturated it contains~$S$. For every $x \in S$ there is a copen section~$T$
  of~$p'$ above~$V$ which passes through~$g(x)$, hence~$x$ is covered by the
  copen section $g^{-1}(T)$. Because~$S$ is compact, there are finitely many
  sections $T_1, \ldots, T_n$ of $p'$ above $V$ such that each $x \in S$ is
  covered by some $g^{-1}(T_i)$. If we let $S_i = S \cap g^{-1}(T_i)$ then
  $S = S_1 \lor \cdots \lor S_n$ and $S_i \subseteq g^{-1}(T_i)$. We have proved
  that $\leqideal{\im{f}}$ is closed under~$\preceq$.

  Conversely, consider a homomorphism $f : A \to A'$ such that
  $\leqideal{\im{f}}$ is closed under~$\preceq$. Suppose $P \subseteq A'$ is a
  prime ideal such that $f^{-1}(P)$ is also a prime ideal. Given any $a' \in A'
  - P$, we need to show that $\up{f}$ is defined at $\elemsk{P}{a'}$. There
  exists $a \in A$ such that $f(a) \not\in P$. By
  Lemma~\ref{lemma:elemsk-cuts-accross}, proved below, there is $b \in A' - P$
  such that $a' \thetaeq_P b$ and $b \Drel f(a)$. It follows that $b\preceq
  f(a)$ and thus $b\in \leqideal{\im{f}}$ by the assumption. Hence there exits
  $a_1,\dots, a_n\in A$ and $b_1,\dots, b_n\in A'$ such that $b=b_1\lor \dots
  \lor b_n$ and $b_i\leq f(a_i)$ for all $i$. Because $b = b_1 \lor \dots \lor
  b_n$ and $b \not\in P$ it follows that not all $b_i$ can lie in $P$. Say $b_1,
  \dots, b_j\not \in P$ while $b_{j+1},\dots, b_n\in P$. Then $b \thetaeq_P b_1
  \lor \dots \lor b_j \thetaeq_P f(a_1 \lor \dots \lor a_j)$ by Lemma
  \ref{lemma:s-less-than-t-not-in-P} and so $\up{f}(\elemsk{P}{a'}) =
  \up{f}(\elemsk{P}{b}) = \elemsk{\down{f}(P)}{a_1 \lor \cdots \lor a_j}$ is
  defined.
\end{proof}

\begin{lemma}
  \label{lemma:elemsk-cuts-accross}
  Let $A$ be a skew algebra and $P \subseteq A$ a prime ideal. For every $a, b
  \in A - P$ there is $c \in A - P$ such that $a \thetaeq_P c$ and $c \Drel b$.
\end{lemma}

\begin{proof}
  If we take $c=(a\land b\land a)\lor b \lor (a\land b\land a)$ then $c\Drel b$
  obviously holds. Next, $a \not\in P$ and $b\not\in P$ together imply $a \land
  b \land a \not\in P$, and $a\land b\land a \thetaeq_P a$ follows by Lemma
  \ref{lemma:s-less-than-t-not-in-P}. Finally, $a\land b\land a\leq c$ and $a\land b\land a \thetaeq_P
  c$ follows, again by Lemma \ref{lemma:s-less-than-t-not-in-P}.

\end{proof}


\section{Lattice sections of skew algebras}
\label{sec:lattice-sections}

A \emph{lattice section} of a skew lattice $A$ is a section $\ell :
\D{A} \to A$ of the canonical projection $q : A \to \D{A}$ which
preserves $0$, $\land$ and $\lor$. We construct a right-handed skew
Boolean algebra without a lattice section. This answers negatively
the open question whether every skew lattice has a section.

\begin{proposition}
  A right-handed skew algebra has a lattice section if, and only if,
  the corresponding skew Boolean space has a global section.
\end{proposition}

\begin{proof}
  Let $A$ be a skew algebra with a lattice section $\ell : \D{A} \to A$. For
  every $d \in \D{A}$ the étale map $q_A : \Sk{A} \to \St{A}$ has a local
  section $s_d : \baseN{d} \to E$ which maps $\baseN{d}$ to $\baseM{\ell(d)}$.
  We can glue the local sections into a global one as long as they are
  compatible. To see that this is the case, take any $d, e \in \D{A}$ and
  compute by Lemma~\ref{prop:baseM-properties}
  \begin{multline*}
    \baseM{\ell(d \land e)} =
    \baseM{\ell(e \land d \land e)} =
    \baseM{\ell(e) \land \ell(d) \land \ell(e)} = \\
    q_A^{-1}(\baseN{d}) \cap \baseM{\ell(e)} =
    q_A^{-1}(\baseN{d} \cap \baseN{e}) \cap \baseM{\ell(e)},
  \end{multline*}
  and similarly
  \begin{equation*}
    \baseM{\ell(e \land d)} =
    \baseM{\ell(d \land e\land d)} =
    q_A^{-1}(\baseN{d} \cap \baseN{e}) \cap \baseM{\ell(d)}.
  \end{equation*}
  Therefore, $s_d$ and $s_e$ restricted to $\baseN{d} \cap \baseN{e}$
  are both equal to $s_{d \land e}$.

  Conversely, suppose $p : E \to B$ is a skew Boolean space with a global
  section $s : B \to E$. Then the map $V \mapsto s(V)$ is a lattice section for
  $A_p$, since $s(\emptyset) = \emptyset$,
  \begin{multline*}
    s(U) \land s(V) =
    p^{-1}(p(s(U))) \cap s(V) = \\
    p^{-1}(U) \cap s(V) =
    s(U) \cap s(V) = s(U \cap V),
  \end{multline*}
  and similarly $s(U) \lor s(V) = s(U \cup V)$.
\end{proof}

\noindent
The preceeding construction of a global section from a lattice section $\ell :
\D{A} \to A$ used only preservation of $0$ and $\land$ by $\ell$. Thus we proved
in passing that a right-handed skew algebra has a lattice section if, and only
if, it has a section which preserves $0$ and $\land$.

We next construct a skew Boolean space which does not have a global section.
Consequently, the corresponding skew algebra does not have a lattice section. To
find a counter-example we first look at a sufficient condition for existence of
global sections.

\begin{proposition}
  A skew Boolean space has a global section if the base space is a countable
  union of compact open sets. Dually, a right-handed skew algebra has a lattice
  section if it contains a cofinal countable chain for to the natural preorder.
\end{proposition}

\begin{proof}
  Suppose the base $B$ of a skew Boolean space $p : E \to B$ of Boolean spaces
  is covered by countably many copen sets. Because a finite union of such sets
  is again copen, there is in fact a countable chain $C_0 \subseteq C_1
  \subseteq \cdots$ of copen sets which cover~$B$. Each $C_i$ has a local
  section $s_i : C_i \to E$. These can be used to form a global section $s : B
  \to E$ which equals $s_i$ on $C_i - C_{i-1}$.

  To transfer the construction to the algebraic side of duality, observe that a
  sequence $c_0, c_1, c_2, \ldots$ in $A$ is a cofinal chain for the natural
  preorder if, and only if, the corresponding sequence of copen sets
  $\baseN{c_0}, \baseN{c_1}, \baseN{c_2}, \ldots$ is a chain that covers
  $\St{A}$.
\end{proof}

Thus we must look for a counter-example whose base space is fairly large. A
simple one is the first uncountable ordinal $\omega_1$ with the interval
topology. Recall that the points of $\omega_1$ are the countable ordinals and
that the interval topology is generated by the open intervals $I_{\alpha, \beta}
= \set{\gamma \in \omega_1 \mid \alpha < \gamma < \beta}$ for $\alpha < \beta <
\omega_1$.
We also need a total space, for which we take the disjoint sum of intervals
\begin{equation*}
  \textstyle
  E = \coprod_{\alpha < \omega_1} [0, \alpha],
\end{equation*}
where each interval $[0, \alpha]$ carries the interval topology. The points of
$E$ are pairs $(\alpha, \beta)$ with $\beta \leq \alpha$ and the basic open sets
are of the form $\set{\alpha} \times I_{\beta, \gamma}$ for $\beta < \gamma \leq
\alpha$. We define the étale map $p : E \to \omega_1$ to be the second
projection, so that the fiber above $\beta$ contains points $(\alpha, \beta)$
with $\beta \leq \alpha < \omega_1$.

A global section $\omega_1 \to E$ of $p$ is a map $\beta \mapsto (s(\beta),
\beta)$ where $s : \omega_1 \to \omega_1$ is a progressive locally constant map.
By progressive we mean that $\beta \leq s(\beta)$ for all $\beta < \omega_1$. To
see that $s$ is locally constant, observe that for any $\beta \in \omega$ the
preimage of the open subset $\set{s(\beta)} \times [0, s(\beta)] \subseteq E$
under the section is open, hence $\beta$ has an open neighborhood which $s$ maps
to $\set{s(\beta)}$.

The map $s$ preserves suprema of monotone sequences because it is locally
constant. Above every $\beta < \omega_1$ there is a fixed point of $s$, namely
the supremum of the monotone sequence
\begin{equation*}
  \beta \leq s(\beta) \leq s(s(\beta)) \leq \cdots
\end{equation*}
By using this fact repeatedly, we obtain a strictly increasing sequence
$\gamma_0 < \gamma_1 < \gamma_2 < \cdots$ of fixed points of~$s$. Their supremum
$\gamma_\infty = \sup_i \gamma_i$ is a fixed point of~$s$, too. Since $s$ is
locally constant there exists $\gamma_i$ such that $s(\gamma_i) = s(\gamma)$,
which yields the contradiction $\gamma_\infty = s(\gamma_\infty) = s(\gamma_i) =
\gamma_i < \gamma_\infty$. We therefore conclude that~$s$ does not exist and $p
: E \to \omega_1$ does not have a global section.


\begin{thebibliography}{10}

\bibitem{BL}
R.~J. Bignall and J.~E. Leech.
\newblock Skew boolean algebras and discriminator varieties.
\newblock {\em Algebra Universalis}, 33:387--398, 1995.

\bibitem{Cor}
W.~H. Cornish.
\newblock Boolean skew algebras.
\newblock {\em Acta Mathematica Academiae Scientiarum Hungarkcae Tomus},
  36:281--291, 1980.

\bibitem{johnstone82:_stone}
P.~T. Johnstone.
\newblock {\em Stone spaces}.
\newblock Cambridge University Press, 1982.

\bibitem{kudryavtseva11:_stone_boolean}
G.~Kudryavtseva.
\newblock A refinement of {S}tone duality to skew {B}oolean algebras.
\newblock \url{http://arxiv.org/abs/1102.1242v2}, 2011.

\bibitem{L1}
J.~Leech.
\newblock Skew lattices in rings.
\newblock {\em Algebra Universalis}, 26:48--72, 1989.

\bibitem{L4}
J.~Leech.
\newblock Normal skew lattices.
\newblock {\em Semigroup forum}, 44:1--8, 1992.

\bibitem{maclane92:_sheav_geomet_logic}
S.~MacLane and I.~Moerdijk.
\newblock {\em Sheaves in Geometry and Logic: A First Introduction to Topos
  Theory}.
\newblock Springer-Verlag, 1992.

\bibitem{Palmgren}
E.~Palmgren and S.~J. Vickers.
\newblock Partial {H}orn logic and cartesian categories.
\newblock {\em Annals of Pure and Applied Logic}, 145(3):314--353, 2007.

\bibitem{stone36:_boolean}
M.~H. Stone.
\newblock The theory of representations for {B}oolean algebras.
\newblock {\em Transactions of the American Mathematical Society},
  40(1):37--111, July 1936.

\bibitem{stone37:_applic_boolean}
M.~H. Stone.
\newblock Applications of the theory of {B}oolean rings to general topology.
\newblock {\em Transactions of the American Mathematical Society},
  41(3):375--481, May 1937.

\end{thebibliography}
\end{document}